\documentclass[11pt,twoside]{article}
\usepackage{amsmath,amsthm}
\usepackage{amssymb,latexsym}
\usepackage{mathrsfs}
\usepackage{amscd}
\usepackage{enumerate}
\usepackage{color}
\usepackage{cite}
\usepackage{amssymb}
\usepackage[mathscr]{euscript}
\usepackage{graphicx}
\usepackage{titletoc}
\usepackage{hyperref}

\titlecontents{section}[0pt]{\addvspace{3pt}\filright}
              {\contentspush{\thecontentslabel\quad}}
              {}{\titlerule*[8pt]{$\cdot$}\contentspage}

\newtheorem{theorem}{Theorem}[section]
\newtheorem{corollary}[theorem]{Corollary}
\newtheorem{lemma}[theorem]{Lemma}
\newtheorem{proposition}[theorem]{Proposition}
\newtheorem{definition}[theorem]{Definition}

\newtheorem{example}[theorem]{Example}

\theoremstyle{definition} \theoremstyle{remark}
\numberwithin{equation}{section}

\newcommand{\ttt}{\mathbf{t}}
\newcommand{\uuu}{\mathbf{u}}
\newcommand{\vvv}{\mathbf{v}}
\newcommand{\www}{\mathbf{w}}
\newcommand{\xxx}{\mathbf{x}}
\newcommand{\yyy}{\mathbf{y}}
\newcommand{\zzz}{\mathbf{z}}

\newcommand{\diam}{{\rm diam}}

\newcommand{\ed}{\mathfrak{E}}

\newcommand{\en}{\mathcal{E}}

\newcommand{\dom}{\mathcal{D}}

\newcommand{\supp}{{\rm supp}}

\newcommand{\ds}{\displaystyle}

\newcommand{\norm}{\mathcal{N}}

\begin{document}

\title{\large {\textbf {RANDOM WALKS AND INDUCED DIRICHLET FORMS \\
ON SELF-SIMILAR SETS}}}

\author{\small {SHI-LEI KONG, KA-SING LAU and TING-KAM LEONARD WONG}}

\date {}
\maketitle

\abstract{Let $K$ be a self-similar set satisfying the open set condition. Following Kaimanovich's elegant idea \cite{Ka}, it has been proved that on the symbolic space $X$ of $K$ a natural augmented tree structure ${\mathfrak E}$ exists; it is hyperbolic, and the hyperbolic boundary $\partial_HX$ with the Gromov metric is H\"older equivalent to $K$. In this paper we consider certain reversible random walks with return ratio $0< \lambda <1$ on $(X, {\mathfrak E})$. We show that the Martin boundary ${\mathcal M}$ can be identified with  $\partial_H X$ and $K$. With this setup and a device of Silverstein \cite {Si}, we obtain precise estimates of the Martin kernel and the Na\"{i}m kernel in terms of the Gromov product. Moreover, the Na\"{i}m kernel turns out to be a jump kernel satisfying the estimate $\Theta (\xi, \eta) \asymp |\xi-\eta|^{-(\alpha+ \beta)}$, where $\alpha$ is the Hausdorff dimension of $K$ and $\beta$ depends on $\lambda$. For suitable $\beta$, the kernel defines a regular non-local Dirichlet form on $K$. This extends the results of Kigami \cite {Ki2} concerning random walks on certain trees with Cantor-type sets as boundaries (see also \cite {BGPW}).}

\tableofcontents

\renewcommand{\thefootnote}{}

\footnote{The research is supported in part by the HKRGC grant and the NNSF of China (no. 11371382).}

\renewcommand{\thefootnote}{\arabic{footnote}}
\setcounter{footnote}{0}

\bigskip

\section{Introduction}
\label{sec:1}

\noindent Let ${\mathbb D}$ be the open unit disk, and let ${\mathbb T}$ be the boundary circle parameterized by  $\{\theta: 0 \leq \theta <2\pi\}$. Let
\begin{equation} \label{eq1.1}
{\mathcal E}_{\mathbb D} (u, v)  = \int_{\mathbb D} \nabla u (x)  \nabla v(x) dx
\end{equation}
be the standard Dirichlet form on ${\mathbb D}$.   In classical analysis, it is well-known that a function $\varphi \in L^1({\mathbb T})$ can be extended to a harmonic functions on ${\mathbb D}$ via the Poisson integral
$$
(H\varphi)(x) = \int_{\mathbb T} \varphi (\theta) K(x, \theta)d \theta,  \qquad x \in {\mathbb D},
$$
where $K(x, \theta)$ is the Poisson kernel. Furthermore, there is an induced Dirichlet form on ${\mathbb T}$ defined by
$$
{\mathcal E}_{\mathbb T} (\varphi, \psi) = {\mathcal E}_{\mathbb D} (H\varphi, H\psi).
$$
Indeed, it can be shown that
\begin{equation} \label{eq1.2}
{\mathcal E}_{\mathbb T} (\varphi, \psi) = \frac 1{16 \pi}\int_ {\mathbb T} \int_ {\mathbb T} (\varphi(\theta) -\varphi (\theta'))(\psi(\theta) -\psi(\theta'))\ \frac 1{\sin^2(\frac {\theta- \theta'}2)}\ d\theta d\theta'.
\end{equation}
This integral is called the {\it Douglas integral} (see \cite[Section 1.2]{FOT}). From the probabilistic point of view, the Dirichlet form in \eqref{eq1.1} is associated with a Brownian motion on ${\mathbb D}$. The hitting distribution of the Brownian motion at the boundary ${\mathbb T}$ (starting from $0$)  is the uniform distribution $\frac {d\theta}{2 \pi}$; the induced Dirichlet form in \eqref{eq1.2} corresponds to the reflecting Brownian motion on $\overline {\mathbb D}$ time-changed by its local time on ${\mathbb T}$, and defines a jump process on ${\mathbb T}$ which is a Cauchy process \cite{CF}.

\medskip

The above consideration has a counterpart in Markov chain theory. Let $\{Z_n\}_{n=0}^\infty$ be a transient Markov chain on an infinite discrete set $X$ with transition probability $P$. According to the discrete potential theory of Markov chains \cite{Dy,Wo1,Wo2}, a chain starting at the reference point $\vartheta$ will converge to the {\it Martin boundary} ${\mathcal M}$ at infinity, and defines a hitting distribution $\nu=\nu_\vartheta$ on ${\mathcal M}$. Also, there is a Martin kernel $K(x, \xi), \  x\in X, \xi \in {\mathcal M}$, which plays the same role as the Poisson kernel: if we define
$$
(H \varphi) (x) = \int_{\mathcal M}\varphi (\xi) K(x, \xi) d\nu (\xi),  \qquad x \in X,
$$
then $u=H\varphi$ is harmonic on $X$, i.e., $u=Pu$.
We call a Markov chain {\it reversible} if the transition probability is of the form $P(x, y) = c(x, y)/m(x)$,  where $c(x,y) =c(y,x) \geq 0$, and $m(x) = \sum_{y\in X} c(x,y)$. We define a graph energy ${\mathcal E}_X$ on $X$ by
\begin{equation} \label{eq1.3}
{\mathcal E}_X[u] = \frac 12\sum_{x, y \in X} c(x, y) |u(x) -u(y)|^2.
\end{equation}
In \cite{Si}, Silverstein showed that for such  Markov chain, there is an energy form ${\mathcal E}_{\mathcal M}$ on ${\mathcal M}$ that satisfies
$$
 {\mathcal E}_{\mathcal M} [\varphi] =  {\mathcal E}_X [H\varphi],
$$
and ${\mathcal E}_{\mathcal M}$ has the expression
\begin{equation}\label{eq1.4}
{\mathcal E}_{\mathcal M} [\varphi] = \int_{\mathcal M}\int_{\mathcal M} |\varphi(\xi) - \varphi (\eta)|^2 \,\Theta (\xi, \eta)\,d\nu (\xi) d\nu (\eta).
\end{equation}
Here $\Theta (\cdot, \cdot)$ is called the {\it Na\"{i}m's $\Theta$-kernel} (or simply, the {\it Na\"{i}m kernel}). It was first introduced in classical potential theory by Na\"{i}m \cite{Na}, and a general Douglas integral formula (corresponding to \eqref{eq1.2}) on Euclidean domain was proved by Doob \cite{Do}. Recently, Georgakopoulos introduced a class of ``group-walk random graphs", and he outlined a study of the Poisson boundary and the Na\"{i}m kernel using electrical network theory in \cite{Ge}.

\bigskip

The domain of ${\mathcal E}_{\mathcal M}$ in \eqref{eq1.4} consists of  square integrable functions $\varphi$ such that ${\mathcal E}_{\mathcal M}[\varphi] < \infty$. If the domain is dense in $L^2({\mathcal M}, \nu)$, then ${\mathcal E}_{\mathcal M}$ defines a non-local Dirichlet form. In the analysis of fractals, there is a large literature on the study of local and non-local Dirichlet forms as well as their associated heat kernels on self-similar sets, $d$-sets and more general metric measure spaces \cite{CF, CK, GHL1, GHL2,GHL3,J,Ki1,Ki2,P1,P2,St,Str}. In many cases, a non-local form can be obtained by subordination of a local form, and has a jump kernel with order $|\xi-\eta|^{-(\alpha + \beta)}$, where $\alpha$ is the Hausdorff dimension of the underlying set and $\beta$ is the walk dimension of the corresponding stable-like process \cite{St,CK}.

\medskip

For a self-similar set $K$ generated by an iterated function system (IFS), there is a symbolic space (coding space) $\Sigma^*$ which gives a convenient representation of any $\xi \in K$ (analogous to the dyadic expansion of a real number). If the IFS satisfies the {\it open set condition} (OSC), then the representation is unique for generic points of $K$. Recently, there are studies of random walks on $\Sigma^*$ such that $K$ can be identified with the Martin boundary ${\mathcal M}$ under the canonical homeomorphism \cite{DS1,DS2,DS3,JLW, Ka, LN1, LN2, LW2}. In particular, Kaimanovich  \cite{Ka} used the Sierpi\'{n}ski gasket to introduce a natural {\it augmented tree} (Sierpi\'{n}ski graph) by adding horizontal edges on $\Sigma^*$ according to the intersections of the cells from the IFS. This work brings into play the hyperbolic structure and hyperbolic boundary which are powerful tool for studying random walks on the graph. Kaimanovich's augmented tree was extended to general self-similar sets in \cite{LW1, LW3}.

\bigskip

In this paper we will study random walks on the augmented trees $(\Sigma^*, {\mathfrak E})$ and their induced Dirichlet forms on $K$. We investigate a class of reversible random walks on $(\Sigma^*,{\mathfrak E})$ so that

\vspace {0.2cm}
\noindent \hspace {0.3cm} (i) the self-similar set $K$ can be identified with the Martin boundary, and

\vspace{0.1cm}

\noindent \hspace {0.3cm}(ii) the hitting distribution $\nu = \nu_\vartheta$ is the normalized Hausdorff measure on $K$, and the Na\"{i}m kernel is of order $|\xi-\eta|^{-(\alpha + \beta)}$.

\medskip

As a special case, Kigami \cite{Ki2} studied reversible random walks on trees where the Martin boundaries are Cantor-type sets. He used the resistance metric to give explicit expressions of the hitting distribution $\nu$ and the Martin kernel $K(x,y)$. Also, under the volume doubling property of $\nu$ with respect to the resistance metric, he studied the associated jump process and estimates of the heat kernel. These results were extended to the non-compact case in \cite{Ki3}. Recently, a duality between reversible random walks on trees and certain jump processes on the boundaries was studied in detail in \cite{BGPW}.

\bigskip

 Let $\{S_i\}_{i=1}^N$ denote an IFS of contractive similitudes, and let $K$ denote the self-similar set generated. For simplicity, here we only state the case where the IFS is {\it homogenous} (all the maps $S_i$ have the same contraction ratio $r$). Following standard notations, we let $\Sigma = \{1, 2, \ldots, N\}$, $\Sigma^n  =
\{\xxx=i_1 \cdots i_n: i_1, \cdots, i_n \in \Sigma\}, \ n \geq 1$
(by convention $\Sigma^0  =\{\vartheta\}$) and use $|\xxx|$ to denote the length of $\xxx$; also we let $\Sigma^{\ast}:=\bigcup_{n=0}^{\infty}\Sigma^n$, the set of finite words, and  $\Sigma^{\infty}:=\{ i_1 i_2 \cdots :\
i_1, i_2, \cdots \in \Sigma\}$, the set of infinite words.
There is a natural surjection $\kappa: \Sigma^{\infty} \rightarrow K$
defined by
\begin{equation*}
\{\kappa(\omega)\} = {\bigcap}_{m \geq 0} \ S_{i_1 i_2 \cdots i_m}(K), \qquad \omega = i_1 i_2 \cdots \in \Sigma^{\infty},
\end{equation*}
where $S_{i_1 i_2 \cdots i_m} = S_{i_1} \circ S_{i_2} \circ \cdots \circ S_{i_m}$. Hence each $\xi \in K$ admits a symbolic representation (coding) $\omega \in \Sigma^\infty$.

\medskip
The symbolic space $\Sigma^*$ has a natural tree structure ${\mathfrak E}_v$ which is referred to as the set of ``vertical edges". To consider (i) above, we enrich this simple graph by adding more ``horizontal edges", denoted by ${\mathfrak E}_h$, to reflect the behaviors of the neighboring cells of $K$ at each level:
\begin {equation} \label{eq1.5}
(\xxx, \yyy) \in {\mathfrak E}_h\ \hbox {if} \  \xxx \neq \yyy, \
|\xxx|=|\yyy| {\it \text{ and }} \inf\limits_{\xi, \eta \in K}|S_{\xxx}(\xi) - S_{\yyy}(\eta)| \leq \gamma \cdot r^{|\xxx|},
\end{equation}
 where $\gamma > 0$ is any fixed constant (see Definition \eqref {de3.3}). We write $X_n = \bigcup_{j=1}^n\Sigma^j$, and $X = \Sigma^*$ for brevity, and
let ${\mathfrak E} = {\mathfrak E}_v \cup {\mathfrak E}_h$. We call $(X, {\mathfrak E})$ the {\it augmented tree} of $K$. This notion was invented by Kaimanovich in \cite{Ka}, where the last condition in \eqref{eq1.5} was replaced by $S_\xxx(K) \cap S_\yyy(K) \neq \emptyset$.  For general IFS (even with the OSC), there are certain technical difficulties to identify $K$ with the graph at infinity [LW1]; the modification in \eqref{eq1.5} is from [LW3] which avoids the superfluous conditions in [LW1].

\medskip

Recall that for a hyperbolic graph $X$, there is a {\it Gromov product} $(x|y)$ for $x, y \in X$, a {\it Gromov metric} $\varrho_a(x, y) = e^{-a(x|y)}$ on $X$, and a {\it hyperbolic boundary} $\partial_HX$ with $\varrho_a(\cdot, \cdot )$ extended to $\partial_HX$ (see Section \ref{sec:2} below for more details).

\medskip

\begin{theorem} \label{th1.1} \hspace{-2mm} {\rm\cite{Ka,LW1,LW3}}~For any IFS the augmented tree $(X,{\mathfrak E})$ is hyperbolic.
Moreover, there is a canonical identification ${\iota}: \partial_H X \rightarrow K$ (independent of $\gamma>0$), such that  $K$ and  $\partial_H X$ are H\"{o}lder equivalent, i.e., $\varrho_a(\xi, \eta) \asymp |{\iota}(\xi)-{\iota}(\eta)|^{-a/\log r}$.
\end{theorem}

\bigskip
For a homogenous IFS $\{S_i\}_{i=1}^N$ satisfying the open set condition (OSC), the self-similar set has Hausdorff dimension $\alpha = \log N/ |\log r|$. For $0<\lambda<1$,  we define a class of reversible random walks on $(X, {\mathfrak E})$  with conductance
$$
c(\xxx, \xxx^-)  = (\lambda^{-1} r^\alpha)^{|\xxx |}, \qquad
\xxx \in X.
$$
and $c(\xxx, \yyy) \asymp c(\xxx, \xxx^-)$ for $(\xxx, \yyy) \in \ed_h$ where $\xxx^-$ is the parent of $\xxx$ (here $\asymp$ means the two terms dominate each other by two constants $C_1$, $C_2$ independent of the variables), and call it a {\it $\lambda$-natural random walk} ($\lambda$-NRW) (see Definition \ref{th4.4}). Here $r^\alpha$ corresponds to the natural weight $p_i = r^\alpha$ of the IFS $\{S_i\}_{i=1}^N$ which generates, as a self-similar measure, the normalized Hausdorff measure ${\mathcal H}^\alpha$ on $K$;  the parameter $\lambda$ is the ratio of the probabilities for the walk to go upward or downward at each $\xxx$ (see \eqref{eq4.1} and Section \ref{sec:4}).
Note that a simple random walk (SRW) has transition probability
$$
P(\xxx,  \yyy) = \frac 1{{\rm deg}(\xxx)} \qquad  \hbox {for}  \quad (x,y) \in \ed,
 $$
 where ${\rm deg}(\xxx)$ is the number of edges joining $\xxx$. It is easy to see that a SRW is a $1/N$-NRW  with $c(\xxx, \xxx^-) = c(\xxx, \yyy) =1$ for $(\xxx,\yyy) \in \ed_h$. Using Theorem \ref{th1.1} and a well-known result of Ancona \cite{An1,An2,Wo1} on uniformly irreducible random walks on hyperbolic graphs, we prove the following theorem (Theorem \ref {th5.1}) which extends \cite[Theorem 4.7]{Ka} for the Sierpi\'{n}ski graph with the simple random walk.
\medskip

\begin{theorem} \label{th1.2} Suppose in addition,  the IFS satisfies the OSC, let $\{Z_n\}_{n=0}^\infty$ be a $\lambda$-NRW on  $(X,{\mathfrak E})$. Then the hyperbolic boundary $\partial_H X$, the Martin boundary $\mathcal{M} $  and the self-similar set $K$ are all homeomorphic under the canonical mapping.
\end{theorem}

\bigskip

The above theorem allows us to identify $K$ with the two boundaries. More importantly, we are able to use the hyperbolic structure to give sharp estimates of the Martin and Na\"{i}m kernels for the $\lambda$-NRW on $(X, {\mathfrak E})$.

\medskip

\begin{theorem} \label{th1.3} The hitting distribution $\nu=\nu_\vartheta$ of the  $\lambda$-NRW on $(X, {\mathfrak E})$ is the normalized Hausdorff measure  ${\mathcal H}^\alpha$ on $K$, and the Martin kernel satisfies the estimate
$$
K (\xxx, \xi)\asymp \lambda^{|\xxx | - (\xxx | \xi)} r^{-\alpha (\xxx |\xi)}, \qquad \xxx \in X, \ \ \xi \in K.
$$
\end{theorem}

\medskip
To derive the hitting distribution, a main part is to show that
$F_n({\vartheta}, \xxx):= {\mathbb P}_{\vartheta}(Z_{\tau_n} = \xxx)$ equals  $r^{\alpha n}$ for $|\xxx | =n$,
where $\tau_n$ is the first hitting time at level $n$ (Theorem \ref{th4.6}). The result then follows from a limiting argument (Theorem \ref{th5.6}). Recall that the Martin kernel is given by
$$
 K(\xxx, \yyy) = \frac {G(\xxx, \yyy) } {G(\vartheta, \yyy) } = \frac {F(\xxx, \yyy) } {F(\vartheta, \yyy) }, \qquad \xxx, \yyy \in X,
$$
where $G(\xxx, \yyy)$ is the Green function, and $F(\xxx, \yyy)$ is the probability of the chain ever visiting $\yyy$ from $\xxx$.
 The estimates of $F(\xxx, \vartheta)$ (Proposition \ref{th4.1}), $F(\vartheta, \xxx)$ (Theorem \ref{th4.6}) and $F(\xxx, \yyy)$ (Theorem \ref{th5.3}) are the core of the proof, they need substantial use of the reversibility of the chain as well as the hyperbolicity of $(X, {\mathfrak E})$.

 \medskip

Following Silverstein \cite{Si}, the Na\"{i}m kernel is defined by
$$
\Theta (\xxx, \yyy) = \frac {K(\xxx, \yyy)}{G(\xxx, \vartheta)} = \frac {K(\xxx, \yyy)}{F(\xxx, \theta) G(\vartheta, \vartheta)}.
$$
It is the discrete analogue of the kernel studied by \cite{Na} in classical potential theory. It is easy to extend  $\Theta (\xxx, \yyy)$ to $\Theta (\xxx, \eta)$ for $\eta \in K$, but the extension to $\Theta (\xi, \eta)$ for $\xi \in K$ is much more involved. In \cite{Na}, the extension involves Cartan's fine topology; in Markov chain theory, Silverstein \cite{Si} proved the identity
$$
\Theta (\xi, \eta) = \lim_{k \to \infty} {\sum}_{\zzz \in X} \ell_k^\xi(\zzz ) \Theta(\zzz, \eta),  \quad \xi,  \eta \in K,
$$
where $\ell_k^\xi(\zzz)$ is the probability for the $\xi$-process to last exit the $k$th-level at $\zzz$ (see \eqref{eq6.1}, \eqref{eq6.3}). Analyzing the above limit, we have (Theorem \ref{th6.4})

\begin{theorem} \label{th1.4}  For the $\lambda$-NRW on $(X, {\mathfrak E})$ with the energy form as in \eqref{eq1.3}, the Na\"{i}m kernel is asymptotically given by
$$
\Theta (\xi, \eta) \asymp (\lambda r^\alpha)^{-(\xi | \eta)}, \quad \xi, \eta \in K.
$$
Consequently by Theorem \ref{th1.1} and \eqref{eq1.4}, the induced energy form satisfies
$$
{\mathcal E}_K[\varphi] \asymp \int_K \int_K |\varphi(\xi) - \varphi (\eta)|^2 \,|\xi - \eta|^{-(\alpha + \beta)} d\nu (\xi) d\nu (\eta)
$$
with $\beta = \log \lambda / \log r$.
\end{theorem}

The domain ${\mathcal D}_K$ of ${\mathcal E}_K$ consists of all functions in $L^2(K, \nu)$ that possess finite $\en_K$-energies, and equals the Besov space $\Lambda^{\alpha, \beta/2}_{2, 2}$ on $K$ \cite{St}. Let $\beta^*$ be the critical exponent such that $\Lambda^{\alpha, \beta/2}_{2,2} \cap C(K)$ is dense in $C(K)$ for all $0 <\beta < \beta^*$,  then the pair $({\mathcal E}_K,{\mathcal D}_K)$ with $\beta \in (0,\beta^*)$ is a non-local Dirichlet form. There are certain values of $\beta$ for which the Dirichlet forms are known to be regular \cite{St, CK, GHL1, GHL2}, and there are some standard examples where the values of $\beta^*$ can be determined explicitly \cite{Ku,J}. These issues will be discussed in Section~\ref{sec:7}, and more detail in the forthcoming paper \cite{KL}.
For a more recent development, Grigor'yan and Yang \cite {GY} reconstructed the corresponding reflective process of the $\lambda$-NRW on the Sierpi\'{n}ski graph and its boundary analogous to the Brownian motion on the closed unit disk, and use this device and Theorem \ref {th1.4} to study the critical exponent $\beta^*$.

\medskip

In the text the theorems are presented in greater generality. We will define a {\it pre-augmented tree} on an $N$-ary tree that has the hyperbolic property (Definition \ref{de3.1}). Many results for the random walks will hold on such graphs. For an IFS $\{S_i\}_{i=1}^N$ that is not homogeneous, we need to re-adjust the levels $\Sigma^n$ as in \eqref{eq3.1''} and restrict to the natural weights $p_i = r_i^\alpha$, where $r_i$ is the contraction ratio of $S_i$, and $\alpha$ satisfies $\sum_{i=1}^N r_i^\alpha =1$. Furthermore, we can set up a more general $\lambda$-{\it quasi-natural random walk} ($\lambda$-QNRW) (Definition \ref{th4.4}), which corresponds to weights $\{p_i\}_{i=1}^N$, that gives a doubling self-similar measure, and Theorems \ref {th1.2} and \ref{th1.3} are still valid (the hitting distribution becomes the self-similar measure defined by the weights $\{p_i\}_{i=1}^N$).

\bigskip

For the organization of the paper, we first recall some basic facts about hyperbolic graphs and the potential theory of Markov chains in Section~\ref{sec:2}. In Section~\ref{sec:3} we introduce the augmented tree $(X, {\mathfrak E})$ on the symbolic space and study its hyperbolic structure. In Section~\ref{sec:4}, we study the classes of $\lambda$-QNRW and $\lambda$-NRW and estimate the associated probabilities. Theorems \ref{th1.2} and \ref{th1.3} are proved in Section~\ref{sec:5}. After a brief review of the Na\"{i}m kernel, we prove Theorem \ref{th1.4} in Section~\ref{sec:6}. Finally, the induced Dirichlet forms will be discussed in Section~\ref{sec:7}.

\bigskip

\section{Preliminaries}
\label{sec:2}

\noindent This section presents a quick summary of hyperbolic graphs and the potential theory of Markov chains we need throughout the paper. For more details, the reader can refer to the excellent monograph \cite{Wo1} by Woess.

\medskip

Let $X$ be a countable set. An {\it (undirected simple) graph} is a pair $(X, {\mathfrak E})$  where ${\mathfrak E}$ is a symmetric subset of $X \times X \setminus \{(x,x):x \in X\}$.  We call $x \in X$ a vertex, and $(x,y) \in {\mathfrak E}$  an edge (denoted by $x \sim y$).  For $x \in X$, the {\it degree} of $x$,
denoted by $\deg (x)$, is the number of edges joining $x$. We say that $(X,{\mathfrak E})$ is
of {\it bounded degree} if
$\sup_{x \in X}\deg (x)<\infty$. For $x,y \in X$, a {\it path} from $x$ to $y$ with
{\it length} $n$ is a finite sequence $[x_0 , x_1, \ldots, x_n]$ such that
$x_0= x, x_n=y$, and $(x_i, x_{i+1}) \in {\mathfrak E}$ for all $0\leq i \leq n-1$.  We always assume that  $(X,{\mathfrak E})$ is {\it connected}, i.e.,  any two vertices can be connected by a path. A graph is called a {\it tree} if any two points can be connected by a unique non-self-intersecting path.

\medskip

We use $\pi (x,y)$ to denote a path connecting two vertices $x$ and $y$ with minimum length, and call it a {\it geodesic path}. The length of $\pi (x,y)$ defines a metric $d(x,y)$ on $X$. Let $\vartheta$ be any fixed vertex in $X$ regarded as the {\it root}. We define the {\it Gromov product} by
\begin{equation} \label{eq2.0}
(x|y) = \frac 12 (|x| + |y| - d(x,y)), \qquad x, y \in X,
\end{equation}
where $|x| = d(\vartheta, x)$.

\medskip

\begin{definition} \label{th2.1} We say that $(X, {\mathfrak E})$ is {\it hyperbolic} if there exists $\delta \geq 0$ such that
\begin{equation}  \label{eq2.1}
(x|y) \geq \min \{(x|z), (z|y)\} -\delta,  \qquad \forall \ x, y, z \in X.
\end{equation}
\end{definition}

\medskip

Clearly, every tree is hyperbolic with $\delta =0$.
For a hyperbolic graph $(X,{\mathfrak E})$, we choose $a > 0$ satisfying $e^{\delta a}< \sqrt{2}$, and define
\begin{equation} \label{eq2.1'}
\varrho_a(x, y) = e^{-a(x|y)}, \ x \neq y,\quad \hbox {and} \quad =0 \ \  \hbox {if}\  x=y.
\end{equation}
This $\varrho_a(\cdot, \cdot)$ on $X \times X$ is equivalent to a metric, and is called {\it Gromov metric} for convenience.

\medskip
\begin{definition} \label{th2.3}
Let $\widehat{X}_H$ denote the completion of  $(X,\varrho_a)$.
The {\it hyperbolic boundary} of $(X,{\mathfrak E})$ is defined as
$\partial_H X = \widehat{X}_H \setminus X$.
\end{definition}

\medskip

The hyperbolic boundary $\partial_H X$ is compact and can be regarded as the collection of all (infinite) geodesic rays starting at $\vartheta$, where two rays
$(x_n)_n$ and $(y_n)_n$ are identified if $\lim_{n \to \infty}(x_n | y_n) = \infty$.

\medskip

Next we consider random walks on a graph $(X,\ed)$. Let $\{Z_n\}_{n=0}^{\infty}$ be a Markov chain on $X$ with transition probability $P$. We write $\mathbb{P}(\cdot \mid Z_0 = x)=\mathbb{P}_x(\cdot)$ for short.
A function $u$ on $X$ is said to be {\it harmonic} (with respect to $P$) if $Pu=u$ where $Pu(x)=\sum_{y\in X} P(x,y)u(y) = \mathbb E_x u(Z_1)$.
The {\it Green function} with respect to $P$ is
$G(x, y) := \sum_{k=0}^{\infty} P^k(x, y)$. We always assume that  the Markov chain is irreducible and transient, i.e., $0 < G(x,y) < \infty$ for any $x,y \in X$.
We define $F(x, y)$ as the probability that the chain starting from $x$ ever visits $y$, i.e.,
\begin{equation} \label{eq2.2}
F(x, y) := \mathbb{P}_x (\exists \ n \geq 0 \ \hbox{such that} \ Z_n = y).
\end{equation}
It is clear that $G(x, y) = F(x, y)G(y, y)$.

\medskip

Fix a reference point $\vartheta \in X$.
We define the {\it Martin kernel} by
\begin{equation*}
K(x,y):=\frac{G(x,y)}{G(\vartheta, y)} = \frac {F(x, y)}{F(\vartheta, y)}, \quad x,y \in X.
\end{equation*}
The following definition is taken from \cite{Wo1}.
\begin{definition} \label{th2.4}
The {\it Martin compactification} of $(X,P)$ is the minimal compactification $\widehat{X}$ of $X$ such that the Martin kernel $K(x, \cdot)$ extends continuously to $\widehat X$ for all $x \in X$. The set $\mathcal{M} = \widehat{X} \setminus X$ is called the {\it Martin boundary} of $(X,P)$.
\end{definition}

It is known that under the Martin topology, $\{Z_n\}_{n=0}^\infty$ converges almost surely to a $\mathcal{M}$-valued random variable $Z_{\infty}$.
For $x \in X$, we define the hitting distribution $\nu_x$ to be the distribution of $Z_{\infty}$ on $\mathcal M$ under $\mathbb P_x$.
The measure $\nu_x$ is absolutely continuous with respect to $\nu_\vartheta$,  and the Radon-Nikodym density is $d \nu_x / d \nu_\vartheta = K(x,\cdot)$.
For $u \in L^1(\mathcal{M},\nu_\vartheta)$, we define its {\it Poisson integral} by
$$
Hu(x) = \mathbb{E}_x u(Z_\infty) = \ds\int_{\mathcal{M}} K(x,\xi)u(\xi)d \nu_\vartheta (\xi), \quad x \in X.
$$
Clearly $Hu$ is harmonic on $X$ since $K(\cdot, \xi)$ is harmonic for all $\xi \in \mathcal{M}$.
A nonnegative harmonic function $h$ on $X$ is called {\it minimal} if $h(\vartheta) = 1$ and if $h' \leq h$ such that $h'$ is nonnegative harmonic functions on $X$, then $h'/h$ is constant. We define the {\it minimal Martin boundary} as $\mathcal{M}_{\min} = \{\xi \in \mathcal{M}: K(\cdot,\xi) \hbox{ is minimal}\}$. Then $Z_\infty \in {\mathcal M}_{\rm min}$ with probability 1, and $\nu_\vartheta$ is supported on ${\mathcal M}_{\rm min}$. We will make use of $\mathcal M_{\min}$ in Section 6 (Lemma \ref{th6.1}).

\bigskip

We say that $P$ has {\it bounded range} if $\sup\{d(x,y): P(x,y)>0 \hbox{ for some } x,y \in X\} < \infty$, and is {\it uniformly irreducible} if there exist $\epsilon > 0$ and $k_0$ such that for any $x \sim y$, there exist $k \leq k_0$ with  $P^k(x,y) \geq \epsilon$. The {\it spectral radius} of $P$ is
 $$
r(P) =  {\limsup}_{n\to \infty} (P^n(x,y))^{1/n} \in (0,1], \quad x,y \in X.
$$
(Note that the limsup is independent of $x$ and $y$ \cite{Wo1}.)
The following important result is due to Ancona \cite{An1,An2}, and the specific version we use is taken from \cite[Theorem 27.1]{Wo1}.

\medskip

\begin{theorem} \label{th2.5} \hspace{-2mm} {\rm (Ancona)} Suppose $(X,\ed)$ is a hyperbolic graph, $P$ is uniformly irreducible with bounded range, and $r(P) < 1$. Then for any $\delta \geq 0$, there is a constant $C_\delta \geq 1$ such that for any $x$, $y \in X$ and $u$ within distance $\delta$ from a geodesic segment between $x$ and $y$,
\begin{equation} \label{eq2.4}
F(x, u)F(u, y) \leq
F(x, y) \leq C_\delta  F(x, u)F(u, y).
\end{equation}

Moreover, the Martin boundary equals the minimal Martin boundary, and is homeomorphic to the hyperbolic boundary.

\end{theorem}

\medskip

We focus on the class of (transient) {\it reversible random walks} on $(X,\mathfrak E)$ where the transition function $P$ satisfies
\begin{equation} \label{eq2.6}
P(x,y) =
\ \dfrac{c(x,y)}{m(x)}, \quad x \sim y,
\end{equation}
and is $0$ elsewhere. Here $c(x,y)=c(y,x)>0$ for each pair $(x,y) \in \mathfrak E$. We call $c(x,y)$ the {\it conductance} of the edge $(x,y)$, and $m(x) = \sum_{y \in X, y \sim x} c(x,y) > 0$ the {\it total conductance} at $x$.
To apply Theorems \ref{th2.5}, we need to check the condition $r(P)<1$. Here we cite a geometric characterization of this condition in \cite[Theorem 10.3]{Wo1} which will be used in Section~\ref{sec:5}. For $A \subset X$, let
$$
\partial A = \{(x,y) \in \mathfrak E: x \in A, \ y \notin A\}
$$
be the {\it boundary} of $A$ and let $c(\partial A) = \sum_{(x,y) \in \partial A}c(x,y)$ (analogous to the surface area of $A$). We also view  $m(A)=\sum_{x \in A}m(x)$ as the ``volume" of $A$.

\medskip

\begin{proposition} \label{th2.7}
Suppose $P$ is a reversible random walk on $(X,\mathfrak E)$. Then $r(P)<1$ if and only if $(X,P)$ satisfies the strong isoperimetric inequality: there exists $\eta >0$ such that
$$
m(A) \leq \eta c(\partial A)
$$
for all finite subsets $A \subset X$.
\end{proposition}

\bigskip

\section{Self-similar sets and augmented trees}
\label{sec:3}

\noindent Let $\{S_i\}_{i=1}^N$, $N \geq 2$, be an {\it iterated function system (IFS)} of contractive similitudes on $\mathbb{R}^d$ where $S_i$ has contraction ratios $r_i \in (0,1)$. The {\it self-similar set} $K$ of the IFS is the unique non-empty compact set in $\mathbb{R}^d$ satisfying $$
K={\bigcup}_{i=1}^{N}S_i(K).
$$
For a set of positive probability weights $\{p_i\}_{i=1}^N$, there is a unique  self-similar measure $\mu$ satisfying the identity
\begin{equation} \label {eq3.1}
\mu (\cdot) = {\sum}_{i=1}^Np_i \mu(S_i^{-1}(\cdot)).
\end{equation}
The IFS is said to satisfy the {\it open set condition} (OSC) if there exists a nonempty bounded open set $U$ such that $S_i(U) \subset U$, and $S_i(U) \cap S_j(U) = \emptyset$ for $i\not = j$. The OSC is the most basic condition imposed on the IFS. It is known that under this condition, the Hausdorff dimension of $K$, denoted by $\dim_H(K)$, equals $\alpha$ where $\sum_{i=1}^N r_i^\alpha=1$. Furthermore, if we take  $p_i= r_i^\alpha, i=1, \cdots, N$,  then the self-similar measure is the normalized $\alpha$-Hausdorff measure on $K$. We call such weights $\{p_i\}_{i=1}^N$ the  {\it natural weight} of the IFS.

\medskip

The IFS gives rise to a symbolic space. Let $\Sigma^*$ and $\Sigma^{\infty}$ be respectively the sets of finite indices (words) and infinite indices as in Section~\ref{sec:1}. For $\xxx=i_1 \cdots i_n \in \Sigma^{\ast}$, we let $S_{\xxx}$ be the composition
$S_{\xxx}=S_{i_1} \circ \cdots \circ S_{i_n}$, and the contraction ratio of $S_\xxx$ is denoted by $r_\xxx = r_{i_1}\cdots r_{i_n}$. There exists a natural surjection $\kappa: \Sigma^{\infty} \rightarrow K$
defined by
\begin{equation*}
\{\kappa(\omega)\} = {\bigcap}_{m \geq 0} S_{i_1 i_2 \cdots i_m}(K), \qquad \omega = i_1 i_2 \cdots \in \Sigma^{\infty}.
\end{equation*}
Hence each $\xi \in K$ admits a symbolic representation (coding) $\omega \in \Sigma^\infty$. Under the OSC, the representation is unique, except a $\mu$-null set, for any self-similar measure $\mu$.

\medskip

The finite word space $\Sigma^\ast$ has a natural tree structure where $\xxx \sim \yyy$ if $\yyy = \xxx i$ or $\xxx = \yyy i$ for some $i\in \Sigma$. Here the empty word $\vartheta$ is the root. The canonical metric on $\Sigma^\ast$ is $\rho_r (\xxx, \yyy)= r^{\xxx\wedge \yyy} $ ($0<r<1$), where $\xxx\wedge\yyy = \min\{k: i_{k+1} \not = j_{k+1}\}$. Note that $\xxx \wedge \yyy$ coincides with the Gromov product, and $\rho_r$ is a visual metric as in \eqref {eq2.1'}; the (hyperbolic) boundary  is $\Sigma^\infty$, and is a Cantor-type set. The symbolic space is a convenient tool to study the self-similar set $K$, but obviously it misses many properties of $K$. Following Kaimanovich's idea \cite{Ka}, we will define an ``augmented tree" by adding more edges to $\Sigma^*$.

\medskip

First, to deal with an IFS with different contraction ratios, we modify the symbolic space by grouping  together  words that have approximately the same contraction ratios. Let $r = \min \{r_i: 1 \leq i \leq N\}$. Define, for $n \geq 1$,
\begin{equation} \label {eq3.1''}
\mathcal{J}_n = \{\xxx = i_1\cdots i_k \in \Sigma^\ast: r_\xxx \leq r^n < r_{i_1 \cdots i_{k-1}}\},
\end{equation}
and $\mathcal{J}_0 = \{\vartheta\}$ by convention. Clearly, $\mathcal{J}_n \cap \mathcal{J}_m = \emptyset$ for $n \neq m$. We  define the {\it modified symbolic space} as $X = \bigcup_{k=0}^\infty\mathcal{J}_k$. We write $|\xxx| = n$ if $\xxx \in \mathcal{J}_n$, and use $\xxx^-$ to denote the {\it parent} of $\xxx \in X$, i.e., the unique word in $X$ such that $|\xxx^-| = |\xxx|-1$. We also define $\xxx^{-k}, k \geq 2$ for the unique word in $X$ such that $|\xxx^{-k}|=|\xxx|-k$ and $\xxx = \xxx^{-k}\zzz$ for some $\zzz \in \Sigma^\ast$.  In particular, if the IFS $\{S_i\}_{i=1}^N$ has equal contraction ratio ({\it homogeneous} IFS), then $\mathcal{J}_n = \Sigma^n$ and $\xxx^{-k}$ is just $\xxx$ with the last $k$ alphabets deleted. The following basic lemma is known.

 \medskip

\begin{lemma} \label{th3.0}  Suppose the OSC holds, and let $\alpha$ be the Hausdorff dimension of $K$. Then
\begin{enumerate}[(i)]
 \item if $\mu$ is the self-similar measure with natural weights $p_i = r_i^\alpha$, then for $\xxx \in \mathcal{J}_n,\  p_\xxx = \mu(S_\xxx(K)) \asymp r^{\alpha n}$;

\item   $r^{-\alpha n} \leq \# \mathcal{J}_n < r^{-\alpha(n+1)}.$
\end{enumerate}
\end{lemma}

\begin{proof} Note that $p_\xxx= r_\xxx^\alpha$, and by definition $r_\xxx \in (r^{n+1}, r^n]$ for $\xxx \in \mathcal{J}_n$. Also the OSC implies that $p_\xxx = \mu(S_\xxx(K))$, and hence (i) follows. Observe that
$$
1=\mu(K) = {\sum\limits}_{\xxx \in \mathcal{J}_n} \mu(S_\xxx(K)) = {\sum\limits}_{\xxx \in \mathcal{J}_n} r_\xxx^\alpha.
$$
As $r_\xxx \in (r^{n+1}, r^n]$, we have $(\# \mathcal{J}_n)r^{\alpha(n+1)} < 1 \leq (\# \mathcal{J}_n)r^{\alpha n}$. This implies (ii).
\end{proof}

 To introduce a graph structure on $X$, we let
$$
{\mathfrak E}_v=\{(\xxx,\yyy) \in X \times X:\
\xxx=\yyy^- \text{ or } \yyy=\xxx^-\},
$$
be the set of edges of the original tree structure on $X$, and use the notation $\xxx \sim_v \yyy$ for $(\xxx,\yyy)\in {\mathfrak E}_v$. We call ${\mathfrak E}_v$ the set of {\it vertical edges} of $X$. Now we augment the tree by adding more {\it horizontal edges}.

\medskip

\begin{definition} \label{de3.1}
Let $X$ be the modified symbolic space associated with $\{S_i\}_{i=1}^N$. We call $(X,{\mathfrak E})$ a {\it pre-augmented tree} if ${\mathfrak E} = {\mathfrak E}_v \cup {\mathfrak E}_h$, where ${\mathfrak E}_h$ satisfies the condition
\begin{equation} \label{eq3.2}
(\xxx,\yyy) \in {\mathfrak E}_h \Rightarrow |\xxx|=|\yyy|, \hbox{ with } \xxx^- = \yyy^- \hbox{ or } (\xxx^-,\yyy^-) \in {\mathfrak E}_h.
\end{equation}
We write $\xxx\sim_h \yyy$ if $(\xxx,\yyy) \in {\mathfrak E}_h$.
\end{definition}

\medskip

For any vertices $\xxx$, $\yyy$ in $(X, {\mathfrak E})$, we say that the geodesic $\pi (\xxx,\yyy)$
is a {\it horizontal geodesic} if it consists of horizontal edges only, and define {\it vertical geodesics} analogously. The geodesic between two vertices $\xxx, \yyy$ may not be unique (see Figure \ref{fig:1}), but there is always a
{\it canonical geodesic}  such that there exist $\mathbf{u}$,
$\mathbf{v} \in \pi (\xxx,\yyy)$ with
\begin{enumerate}[(i)]
\item $\pi (\xxx, \yyy) = \pi (\xxx, \mathbf{u}) \cup
\pi (\mathbf{u}, \mathbf{v}) \cup \pi (\mathbf{v}, \yyy)$,
where $\pi (\mathbf{u}, \mathbf{v})$ is a horizontal geodesic and
$\pi (\xxx, \mathbf{u})$, $\pi (\mathbf{v}, \yyy)$
are vertical geodesics.

\item for any geodesic $\pi' (\xxx, \yyy)$,
$\text{dist} (\vartheta, \pi (\xxx, \yyy)) \leq
\text{dist} (\vartheta, \pi' (\xxx, \yyy))$.
\end{enumerate}

\begin{figure}[ht]
\begin{center}
\includegraphics[width=39mm,height=30mm]{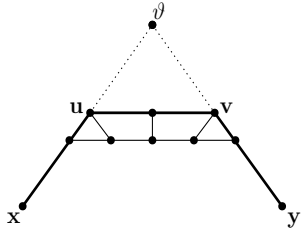}
\caption{{\bf} Canonical geodesic} \label{fig:1}
\end{center}
\end{figure}

If we let $\ell = {\rm dist} (\vartheta, \pi (\bf u, \bf v))$, and $h= d(\bf u , \bf v)$, then the canonical geodesic has a simple geometric interpretation for the Gromov product \eqref{eq2.0},
\begin{equation} \label{eq3.1'}
(\xxx|\yyy) = \ell - \frac 12 h .
\end{equation}
We will use this expression frequently. As a consequence we have (see also \cite [Theorem 3.13] {Ka} for the ``no big squares" condition)

\medskip

\begin{proposition} \label{th3.2} \hspace{-2mm} {\rm \cite[Theorem 2.3]{LW1}} A pre-augmented tree $(X,{\mathfrak E})$ is hyperbolic if and only if there exists $M>0$ such that the lengths of all horizontal geodesics are bounded by $M$.
\end{proposition}

\medskip

The concept of pre-augmented tree is rather flexible (see Remark 2 in the following). For our purpose, we will use the following specific horizontal edge set.

\medskip

\begin{definition} \label{de3.3} \hspace{-2mm} {\rm\cite{LW3}}~For $\gamma > 0$,  we define a horizontal edge set ${\mathfrak E}_h^\gamma$ on $X$ by
\begin{equation} \label{eq3.3}
{\mathfrak E}_h^\gamma=\{(\xxx,\yyy) \in X \times X:\  \xxx \neq \yyy, \
|\xxx|=|\yyy| {\it \text{ and }} \inf\limits_{\xi, \eta \in K}|S_{\xxx}(\xi) - S_{\yyy}(\eta)| \leq \gamma \cdot r^{|\xxx|}\}.
\end{equation}
The {\it augmented tree} $(X,{\mathfrak E}^\gamma)$ is the graph with edge set ${\mathfrak E}^\gamma={\mathfrak E}_v \cup {\mathfrak E}_h^\gamma$.
\end{definition}

\medskip
It is direct to check that $(X,{\mathfrak E}^\gamma)$ is a pre-augmented tree. As the constant $\gamma>0$ has no significance on the edges in the levels ${\mathcal J}_n$ when $n$ is large, we will omit the superscript $\gamma$ in the notations when there is no confusion.  The following theorem is the main reason for introducing the augmented tree.

\medskip

\begin{theorem} \label{th3.6} \hspace{-2mm} {\rm\cite{Ka,LW1,LW3}}~Let $\{S_i\}_{i=1}^N$ be an IFS of contractive similitudes and let $K$ be the self-similar set. Let $(X,{\mathfrak E})$ be the augmented tree. Then
\begin{enumerate}[(i)]
\item $(X,{\mathfrak E})$ is hyperbolic;

\item there is a canonical identification ${\iota}: \partial_H X \rightarrow K$ (independent of $\gamma>0$) defined by $\{{\iota}(\xi)\} = \bigcap_n S_{\xxx_n}(K)$ and $(\xxx_n)_n$ is a geodesic ray converging to $\xi$.  Under this map, $K$ is H\"{o}lder equivalent to $\partial_H X$ in the sense that $\varrho_a(\xi, \eta) \asymp |{\iota}(\xi)-{\iota}(\eta)|^{-a/ \log r}$.
\end{enumerate}
\end{theorem}

\medskip

\noindent {\it Remark 1.}
 Let  $e_0 =0$ and $e_i$, $1\leq i \leq d$, be the standard basis vectors in ${\mathbb R}^d$. Let $\{S_i\}_{i=0}^d$ be an IFS on ${\mathbb R}^d$ defined by
$$
S_i(\xi) = e_i + \frac{1}{2}(\xi-e_i), \quad  \xi \in {\mathbb{R}}^d.
$$
 Then the self-similar set $K$ is the {\it $d$-dimensional Sierpi\'{n}ski gasket}. When $d = 1$, $K$ is simply the unit interval;  when $d=2$, $K$ is the standard Sierpi\'{n}ski gasket. Theorem \ref{th3.6} was first proved in \cite[Section~3]{Ka} for the Sierpi\'{n}ski gaskets with a slightly different  horizontal edge set $\ed_h^0$ (see Section \ref{sec:1}; call this augmented tree the Sierpi\'{n}ski graph). The extension to general self-similar sets was in \cite{LW1}, where we need to assume, in addition, the OSC and some geometric condition on $K$. For many standard cases, ${\mathfrak E}^\gamma_h = {\mathfrak E}^0_h$ for small $\gamma$, but there are examples where they are different for any $\gamma >0$. The setup in Definition \ref{de3.3} with $\gamma >0$ allows us to avoid the delicate behavior on $S_{\xxx}(K) \cap S_{\yyy}(K)$ without changing the boundary. Consequently, in \cite{LW3}, the technical assumptions in \cite{LW1} were eliminated, and Theorem \ref{th3.6} was stated in the most general setting.

\bigskip

\noindent {\it Remark 2.} A pre-augmented tree can be constructed easily, and can be trivial; for example, we can connect all the vertices in each level by edges and still form a hyperbolic graph. On the other hand, if we augment the tree by suitable edges, we can obtain other interesting graphs. For example, starting with the binary augmented tree of the interval ($d=1$ in Remark 1), we can add one more horizontal edge on each level to connect the two end vertices. This new graph is a pre-augmented tree and is hyperbolic (see Figure~\ref{fig:3} or \cite{Ne}).  It is easy to modify the proof of Theorem \ref{th3.6} and show that the hyperbolic boundary is H\"older equivalent to the unit circle.

\begin{figure}[ht]
\begin{center}
\includegraphics[width=39mm,height=39mm]{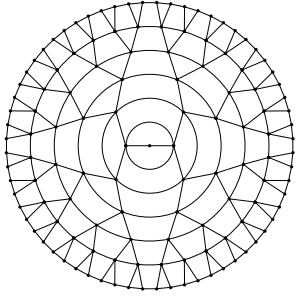}
\caption{{\bf} The circle as a hyperbolic boundary}
\label{fig:3}
\end{center}
\end{figure}

\medskip

That an augmented tree has bounded degree is important when we consider random walks on it. This property is verified when the IFS has the OSC.

\medskip

\begin{proposition} \label{th3.7} \hspace{-2mm} {\rm \cite[Theorem 1.4]{LW3}}~The IFS $\{S_i\}_{i=1}^N$ satisfies the OSC if and only if the augmented tree $(X, {\mathfrak E})$ has bounded degree.
\end{proposition}

\medskip

To conclude this section, we will use the special form of the canonical geodesic to provide more information of the Gromov product in the augmented tree, which will be needed in Section~\ref{sec:6}. Note that $\xxx_n \in {\mathcal J}_n$ for any geodesic ray  $(\xxx_n)_{n=0}^\infty$ in $X$ (starting from $\vartheta)$.

\medskip

\begin{lemma} \label{th3.8}
Let $(X, {\mathfrak E})$ be a pre-augmented tree as in Definition \ref{de3.1}. Let $\zzz \in X$, and let  $(\xxx_n)_{n=0}^\infty$ and $(\yyy_n)_{n=0}^\infty$ be two  distinct geodesic rays from $\vartheta$. Then we have

\begin {enumerate}[(i)]

\item  for $n\geq |\zzz|:=k$, $d(\xxx_n, \zzz) - |\xxx_n| = d(\xxx_{k}, \zzz) - k$;

\item   $\{(\xxx_n|\zzz)\}_n$ is increasing, and
$
(\xxx_n|\zzz) = (\xxx_k|\zzz)$ for  $n \geq k$;

\item  $\{(\xxx_n|\yyy_n)\}_n$ is increasing, and there exists $\ell$ such that
$ (\xxx_n|\yyy_n) = (\xxx_\ell|\yyy_\ell)$ for $n\geq \ell$.
\end{enumerate}
\end{lemma}

\begin{proof}\  (i) For $n \geq |\zzz|:=k$, consider the canonical geodesic $\pi(\xxx_n, \zzz) = \pi(\xxx_n, {\bf u}) \cup \pi({\bf u},{\bf v}) \cup \pi({\bf v},\zzz)$, where $\pi({\bf u}, {\bf v})$ is a horizontal geodesic, and $\pi(\xxx_n, {\bf u})$,  $\pi({\bf v} ,\zzz)$ are vertical geodesics.
Since $|\xxx_n|= n \geq |k| \geq |{\bf u}|=|{\bf v}|$,
$\xxx_{k}$ lies on the vertical geodesic $\pi(\xxx_n, {\bf u})$. Hence (i) follows from
\begin{align*}
d(\xxx_n,\zzz)-|\xxx_n|
%= d(\xxx_n, {\bf u})+ d({\bf u},\zzz)-|\xxx_n|
= d({\bf u},\zzz)-|{\bf u}|
= d({\bf u},\zzz)+ (d(\xxx_k, {\bf u})-|\xxx_{k}|)= d(\xxx_{k}, \zzz)-k .
\end{align*}
\indent(ii) That $\{(\xxx_n|\zzz)\}$ is increasing in $n$ follows by checking the definition of the Gromov product.
Then by (i), we have
\begin{align*}
(\xxx_n|\zzz) = \dfrac{1}{2} (|\zzz|+|\xxx_n|-d(\xxx_n,\zzz))
= \dfrac{1}{2} (2k-d(\xxx_k,\zzz))
= (\xxx_k|\zzz),  \quad n\geq k.
\end{align*}
\indent(iii)\, By the same proof as the above, we see that $\{(\xxx_n|\yyy_n)\}_n$ is increasing.
Let $M$ be the maximal length of the horizontal geodesics in $(X, {\mathfrak E} )$ (by Proposition \ref{th3.2}). Since the two rays are distinct, we have $\lim_{n\to \infty} (\xxx_n|\yyy_n) < \infty$. Then there exists a positive integer $\ell$ such that $d(\xxx_n,\yyy_n) > M$ for any $n > \ell$.
That canonical geodesic $\pi(\xxx_n,\yyy_n)$ is not horizontal means both $\xxx_{n-1}$ and $\yyy_{n-1}$ lie in the two vertical segments of $\pi(\xxx_n,\yyy_n)$ respectively. Now (iii) follows since
$$
(\xxx_n|\yyy_n) = 2n-d(\xxx_n,\yyy_n) = 2(n-1)-d(\xxx_{n-1},\yyy_{n-1}) = (\xxx_{n-1}|\yyy_{n-1}), \quad n > \ell.
$$
\end{proof}

\begin {corollary} \label{th3.9} Let $(X, {\mathfrak E})$ be an augmented tree as in Definition \ref{de3.3}. Let $\xi, \eta \in \partial_H X$,  $\xi \neq \eta$, and let $(\xxx_n), (\xxx'_n), (\yyy_n), (\yyy'_n)$ be geodesic rays, where $(\xxx_n), (\xxx'_n)$ converge to $\xi$, and $(\yyy_n), (\yyy'_n)$ converge to $\eta$. Then for any $\zzz \in X$,
\begin{equation*}
\big |\lim_{n \to \infty} (\xxx_n|\zzz) - \lim_{n \to \infty} (\xxx_n'|\zzz)\big| \leq \dfrac{1}{2},
\quad \hbox {and} \quad
\big|\lim_{n \to \infty} (\xxx_n|\yyy_n) - \lim_{n \to \infty} (\xxx_n'|\yyy_n')\big| \leq 1.
\end{equation*}
\end{corollary}

\begin{proof} We first observe from Theorem \ref{th3.6} that $\xi \in S_{\xxx_n} (K) \cap S_{\xxx'_n}(K)$. Thus $S_{\xxx_n} (K) \cap S_{\xxx'_n}(K) \neq \emptyset$, and by \eqref{eq3.3} either $\xxx_n=\xxx'_n$ or $\xxx_n \sim_h \xxx'_n$. It follows that $d(\xxx_n,\xxx'_n) \leq 1$ for all $n$.
Then the triangle inequality implies
\begin {equation} \label {eq3.5'}
|(\xxx_n|\zzz)-(\xxx_n'|\zzz)| = \dfrac{1}{2}|d(\xxx_n',\zzz)-d(\xxx_n,\zzz)|
\leq \dfrac{1}{2}d(\xxx_n',\xxx_n) \leq \dfrac{1}{2}.
\end{equation}
For the second part, we need only to replace the above $\zzz$ by $\yyy_n$ and $\yyy_n'$ respectively and apply the triangle inequality.
\end{proof}

\medskip

From the corollary, we can extend the Gromov product to $\partial_H X$ by
\begin{equation} \label{eq3.4}
(\xi|\zzz) = \sup \big  \{\lim_{n \to \infty} (\xxx_n|\zzz)\big \},  \qquad \xi \in \partial_H X,  \ \ \zzz \in X,
\end{equation}
and
\begin{equation} \label{eq3.5}
(\xi|\eta) = \sup \big \{\lim_{n \to \infty} (\xxx_n|\yyy_n)\big \},  \qquad \xi, \eta \in \partial_H X,
\end{equation}
where the supremum is taken over all geodesic rays $(\xxx_n)_n$ and $(\yyy_n)_n$ that converge to $\xi$ and $\eta$ respectively.

\bigskip

\section{Constant return ratio and quasi-natural RW}
\label{sec:4}

%\noindent In this section, we define and study certain reversible random walks on the pre-augmented trees. Our main conclusions are some basic estimates of the probabilities $F(\xxx,\vartheta)$ and $F(\vartheta, \xxx)$.

\medskip

\noindent Let $\{Z_n\}_{n=0}^\infty$ be a reversible random walk  on the pre-augmented tree $(X,{\mathfrak E})$. For $\xxx \in X \setminus \{\vartheta\}$, we define the {\it return ratio} at $\xxx$ by
\begin{equation} \label{eq4.1}
\lambda(\xxx) = \dfrac{P(\xxx,\xxx^-)}{\sum\limits_{\yyy:\yyy^- = \xxx}P(\xxx,\yyy)} = \dfrac{c(\xxx,\xxx^-)}{\sum\limits_{\yyy:\yyy^- = \xxx}c(\xxx,\yyy)}.
\end{equation}
For $\lambda >0$, we introduce the following condition on the transition probability $P$:

\medskip

\noindent  $(R_\lambda)$ \ ({\it Constant return ratio}) For any $\xxx \in X \setminus \{\vartheta\}$, $\lambda(\xxx) \equiv \lambda$ is a constant.

\medskip

\noindent For example, the SRW on the augmented tree of a homogeneous IFS $\{S_i\}_{i=1}^N$ satisfies  condition $(R_\lambda)$ with $\lambda = N^{-1}$.

\medskip

For a  fixed level $m \geq 1$, let $X_m := \bigcup_{k=0}^{m}\mathcal{J}_k$ and ${\mathfrak E}_m := {\mathfrak E}|_{X_m \times X_m}$.
Consider the following random walk $\{Z_n^{(m)}\}_{n=0}^{\infty}$ with transition probability
$P_m$ on the graph $(X_m, {\mathfrak E}_m)$: for $\xxx$, $\yyy \in X_m$,
\begin{equation*}
P_m(\xxx, \yyy) = \begin{cases}
P(\xxx,\yyy), \quad \ \ & \hbox{if } \xxx \sim \yyy
\text{ and } |\xxx|<m, \\
\quad 1, \quad \ \ & \text{if } \xxx = \yyy
\text{ and } |\xxx|=m, \\
\quad 0, \quad \ \ & \text{otherwise}.
\end{cases}
\end{equation*}
This is the restriction of $P$ on $X_m \setminus \mathcal{J}_m$, where the vertices in $\mathcal{J}_m$ are absorbing states of $P_m$. Let
$F_m(\xxx, \yyy)$ be the probability of ever visiting $\yyy$ from $\xxx$ in $X_m$, i.e.,
\begin{equation*}
F_m(\xxx, \yyy) := \mathbb{P} (\exists \ n \geq 0 \ \hbox{such that}
\ Z_n^{(m)} = \yyy  \mid  Z_0^{(m)}=\xxx).
\end{equation*}

\medskip

We begin by finding the value of $F_m(\xxx,\vartheta)$ for a reversible random walk with condition $(R_\lambda)$, which is the most basic identity to be used  in this section and Theorem \ref{th5.3}.

\medskip

\begin{proposition}  \label{th4.1}
Let $\{Z_n\}_{n=0}^\infty$ be a reversible random walk on a pre-augmented tree $(X,{\mathfrak E})$ that satisfies $(R_\lambda)$ for some $\lambda>0$. Then for $m \geq 1$ and $\xxx \in X_m$,
\begin{equation} \label{eq4.2}
F_m(\xxx, \vartheta) =
\begin{cases}
\ \dfrac{\lambda^{|\xxx|}-\lambda^m}{1-\lambda^m}, & \quad \hbox{if } \lambda \neq 1, \vspace{0.3cm} \\
\ \dfrac{m-|\xxx|}{m}, & \quad \hbox{if } \lambda = 1.
\end{cases}
\end{equation}
Consequently, $F(\xxx, \vartheta) = \lambda^{|\xxx|}$ if $0\leq \lambda <1$, and $=1$ if $\lambda \geq 1$.
\end{proposition}

\noindent {\it Remark.}  We thank Professor J. Kigami for informing us for the following proof which shortened the original one.

\begin{proof}
For $n \geq 0$, let $t(n) = \inf\{k \geq n: |Z_k| \neq |Z_n|\}$ be the first time that the random walk jumps to a different level from time $n$. Define a sequence of stopping times $\{n_k\}_{k=0}^\infty$ by letting $n_0=0$ and $n_k=t(n_{k-1})$ for $k \geq 1$. For $k \geq 0$, let $L_k=|Z_{n_k}|$ denote the level of the chain.
It can be checked directly that $\{L_k\}_{k=0}^\infty$ is a birth and death chain on the nonnegative integers with the following transition probabilities: $P_L(0,1)=1$, $P_L(\ell,\ell-1)=\lambda/(1+\lambda)$, and $P_L(\ell,\ell+1)=1/(1+\lambda)$ for $\ell \geq 1$. Let $T_\ell$ be the first time that $L_k$ visits level $\ell$, then the expression \eqref{eq4.2} of $F_m(\xxx, \vartheta)=\mathbb P(T_0<T_m \mid L_0=|\xxx|)$ for $\xxx \in X_m$ follows from standard calculations for birth and death chains.

 Observe that for $|\xxx|< m$,  $F_m(\xxx, \vartheta) \nearrow F(\xxx, \vartheta)$ as $m\to \infty$, the second part follows by taking limit of \eqref {eq4.2}.
\end{proof}

\medskip
%
%To illustrate Theorem \ref{th4.1}, we consider the SRW on the binary augmented tree for the unit interval. The values of $F_m(\xxx, \vartheta)$ are calculated for two steps (see Figure~\ref{fig:4}). It is interesting to see that on each level, the probability $F_m(\xxx, \vartheta)$ has the same value even though $\deg(\xxx)$ varies.
%\begin{figure}[ht]
%\begin{center}
%\includegraphics[width=39mm,height=30mm]{Fig4.eps}
%\caption{{\bf} $F_m(\xxx, \vartheta), m=3$, for SRW on the binary augmented tree} \label{fig:4}
%\end{center}
%\end{figure}

 Since our main interest is on transient random walks, we will assume $\lambda \in (0,1)$ throughout the paper. As a simple consequence of Proposition \ref{th4.1} we have the follow result on the Green function, which will be needed to consider $F_m(\vartheta, \xxx)$ and $F(\vartheta, \xxx)$.

\medskip

\begin{lemma} \label{th4.2} Let $\{Z_n\}_{n=0}^\infty$ be a reversible random walk on a pre-augmented tree $(X,{\mathfrak E})$, and $G_m(\cdot, \cdot)$ be the Green function of $\{Z_n^{(m)}\}_{n=0}^\infty$. If $\{Z_n\}_{n=0}^\infty$ satisfies condition $(R_\lambda)$ with $\lambda \in (0,1)$, then $G_m (\vartheta, \vartheta) = \frac{1-\lambda^m}{1-\lambda}$ and $G(\vartheta, \vartheta)= \frac 1{1-\lambda}$.
\end{lemma}

\begin{proof} Note that
$
G_m(\vartheta, \vartheta) = \frac{1}{1-U_m(\vartheta, \vartheta)},
$
where $U_m(\vartheta, \vartheta) := {\mathbb P} (\exists\,n \geq 1 \hbox{ such that } Z^{(m)}_n = \vartheta \mid Z^{(m)}_0 = \vartheta)$,
the probability that the random walk returns to $\vartheta$ after
starting at the root $\vartheta$).  The lemma follows from this identity together with  the one-step formula $U_m(\vartheta, \vartheta) = \sum_{\xxx \in \mathcal{J}_1} P_m(\vartheta, \xxx) F_m(\xxx, \vartheta)$  and Proposition \ref{th4.1}.
\end{proof}

\medskip

\begin{proposition} \label{th4.3}
Let $\{Z_n\}_{n=0}^\infty$ be a reversible random walk on a pre-augmented tree $(X,{\mathfrak E})$ that satisfies condition $(R_\lambda)$. Then for $m \geq 1$ and $\xxx \in \mathcal{J}_m$,
\begin{equation} \label {eq4.5}
F_m (\vartheta, \xxx) = \frac{c(\xxx^-,\xxx)\lambda^{m-1}}{m(\vartheta)}\ .
\end{equation}
\end{proposition}

\begin{proof}
For $\xxx \in {\mathcal J}_m$, using the reversibility $m(\vartheta)G_m(\vartheta,\xxx^-) = m(\xxx^-)G_m(\xxx^-,\vartheta)$, together with Proposition \ref{th4.1} and Lemma \ref{th4.2}, we can evaluate $F_m (\vartheta, \xxx) = G_m(\vartheta,\xxx^-)P(\xxx^-,\xxx)$ directly to get \eqref{eq4.5}.
\end{proof}

\medskip

In view of the expression of $F_m(\vartheta, \xxx)$ in \eqref{eq4.5}, we will introduce a class of conductance $c(x,y)$ so that the limits exist as $m \to \infty$.  (We can not take the limit directly, as $\xxx \in {\mathcal J}_m$ depends on $m$.)

\medskip

\begin{definition} \label{th4.4}
 Let $(X,{\mathfrak E})$ be a pre-augmented tree of an IFS $\{S_i\}_{j=1}^N$. A reversible random walk $\{Z_n\}_{n=0}^\infty$ on $(X,{\mathfrak E})$ with conductance $c:X \times X \rightarrow [0,\infty)$ is called {\it quasi-natural} with return ratio $\lambda \in (0,1)$ ($\lambda$-QNRW) if for a set of probability weights $\{p_i\}_{i=1}^N$,
\begin{enumerate}[(i)]
\item
$c(\xxx,\xxx^-) = p_\xxx\lambda^{-m}$, for $\xxx \in \mathcal J_m$, $m \geq 1$, where $p_\xxx=p_{i_1} \cdots p_{i_k}$ if $\xxx = i_1 \cdots i_k$,

\item
$c(\xxx, \yyy) \asymp c(\xxx,\xxx^-)$ for  $\yyy \sim_h \xxx \in \mathcal J_m$, $m \geq 1$, and
\item
$c(\xxx, \yyy) =0 $ if \  $\yyy \not \sim  \xxx $.
\end{enumerate}
(Here the bounds of $\asymp$ are independent of  $\mathcal J_m$, $m\geq 1$).
Furthermore, if the IFS satisfies the OSC and $p_i = r_i^\alpha$, the natural weights of the IFS, we call $\{Z_n\}_{n=0}^\infty$ a {\it natural random walk} with return ratio $\lambda$ ($\lambda$-NRW) on $(X, {\mathfrak E})$.
\end{definition}

\medskip

For a homogeneous IFS, using $\Sigma^n = {\mathcal J}_n$, it is direct to check from $P(\xxx, \yyy) = \frac {c(\xxx,\yyy)}{m(\xxx)}, \ \xxx\sim \yyy$ that
$P(\xxx, \xxx^-) = C_\xxx \lambda$, and $P(\xxx, \xxx i) = C_\xxx p_i$ for some $C_\xxx>0$. Clearly the walk has constant return ratio $\lambda$. More generally, we have

\medskip

\begin{proposition} \label{th4.5}
Let $\{Z_n\}_{n=0}^\infty$  be a $\lambda$-QNRW on the pre-augmented tree $(X, {\mathfrak E})$. Then $\{Z_n\}_{n=0}^\infty$ satisfies condition $(R_\lambda)$. Also $m(\vartheta) = \lambda^{-1}$, and if $(X, {\mathfrak E})$ is of bounded degree, then
\begin{equation} \label{eq4.6}
m(\xxx) \asymp  c(\xxx, \xxx^{-}), \qquad \xxx \in X.
\end{equation}
\end{proposition}

\begin{proof}
For $\xxx \in X \setminus \{\vartheta\}$, by \eqref{eq4.1}, we have $
\lambda(\xxx)  = {\lambda p_\xxx}/ ({\sum_{\yyy:\yyy^- =\xxx}p_\yyy}) = \lambda.
$
Hence the walk satisfies condition $(R_\lambda)$.  Also,  $m(\vartheta) = \sum_{\yyy \in \mathcal J_1}p_\yyy \lambda^{-1} =\lambda^{-1}$, and \eqref{eq4.6} is straightforward by the bounded degree assumption.
\end{proof}

\medskip

We now apply the previous results to make the following conclusion.

\medskip

\begin{theorem} \label{th4.6}
Let $\{Z_n\}_{n=0}^\infty$ be a $\lambda$-QNRW on a pre-augmented tree $(X, {\mathfrak E})$ with bounded degree. Then
\begin{equation} \label{eq4.7}
F_m(\vartheta, \xxx)  = p_\xxx ,  \qquad \xxx \in {\mathcal J}_m,
\end{equation}
Furthermore,  $F(\vartheta, \xxx) \asymp p_\xxx$ (the bounds depend on $\lambda$) for any $\xxx \in X$.
\end{theorem}

\begin{proof} For $m\geq 1$ and for $\xxx \in {\mathcal J}_m$, Propositions \ref {th4.3} and \ref{th4.5}  imply that
$$
F_m(\vartheta, \xxx) = \frac {c(\xxx, \xxx^-)\lambda^{m-1}}{m (\vartheta)} = \frac {p_\xxx \lambda^{-m}\cdot \lambda^{m-1}}{ \lambda^{-1}} = p_\xxx .
$$
Note that we cannot directly send $m \to \infty$ as $\xxx \in {\mathcal J}_m$ depends on $m$. Instead,  we observe that
$$
F(\vartheta, \xxx) \geq F_m(\vartheta, \xxx) = p_\xxx.
$$
Also, note that $P$ is reversible, and hence
$m(\xxx)G(\xxx, \vartheta) = m(\vartheta)G(\vartheta, \xxx)$ for any vertex $\xxx \in X$.
Hence, for $\xxx \in X \setminus \{\vartheta\}$, by Proposition \ref{th4.1}, Lemma \ref{th4.2} and  Proposition  \ref{th4.5},  we have
\begin{align*} \label{eqM1}
F(\vartheta, \xxx) & = \dfrac{G(\vartheta, \xxx)}{G(\xxx, \xxx)} \ \leq\  G(\vartheta, \xxx) \hspace{1.2cm} (\hbox{by } G(\xxx,\xxx) \geq 1) \\
&= \dfrac{m(\xxx)}{m(\vartheta)}G(\xxx, \vartheta) \ \leq \ C_1 p_\xxx\lambda^{-|\xxx|} F(\xxx, \vartheta)G(\vartheta, \vartheta) \\
&= C_1 p_\xxx \lambda^{-|\xxx|} \cdot \Big(\lambda^{|\xxx|}\cdot \frac{1}{1-\lambda}\Big)
\ =\ \frac {C_1}{1-\lambda} \ p_\xxx.
\end{align*}
This shows that $F(\vartheta, \xxx ) \asymp p_\xxx$.
\end{proof}

\medskip

\noindent {\it Remark 1}. In general, we cannot expect a $\lambda$-QNRW to have the {\it strict reversibility} property (i.e., $0< M^{-1} \leq c(\xxx, \yyy)\leq M$ for any $\xxx \sim \yyy$). More precisely, we can show that if the IFS satisfies the OSC, then strict reversibility of $\lambda$-QNRW implies that  $\lambda =r^\alpha$, where $r = \min_{1\leq i\leq N} r_i$.
\vspace {0.1cm}

Indeed, suppose strict reversibility holds, and let $c(\mathcal J_m) = \sum_{\xxx \in \mathcal J_{m}} c(\xxx,\xxx^-)$, $m \geq 0$. Then by \eqref{eq4.1}, we have $c( {\mathcal J}_{m+1}) = \lambda^{-m}c(\mathcal J_1) = \lambda^{-m}m(\vartheta)$. Recall that $r^{-\alpha m}\leq \#{\mathcal J}_m < r^{-\alpha(m+1)}$ where $\alpha$ satisfies $\sum_{i=1}^N r_i^\alpha =1$ (Lemma \ref{th3.0}). Hence
$$
0 < M^{-1} \leq \dfrac{c(\mathcal J_{m+1})}{\# \mathcal J_{m+1}} \leq \dfrac{\lambda^{-m}m(\vartheta)}{r^{-\alpha (m+1)}}, \qquad  m\geq 0.
$$
This implies $\lambda \leq r^\alpha$. On the other hand,
$$
M\geq \dfrac{c(\mathcal J_{m+1})}{\# \mathcal J_{m+1}} \geq \dfrac{\lambda^{-m}m(\vartheta)}{r^{-\alpha (m+2)}}, \qquad  m\geq 0
$$
 yields $\lambda \geq r^\alpha$, and hence $\lambda = r^\alpha$. It follows that  in Definition \ref {th4.4}(i), we must have $p_i = r_i^\alpha$, the natural weight to have the strict reversibility property.

\bigskip

\noindent {\it Remark 2}. Conditions (i) and (ii) of the $\lambda$-QNRW imply that $c(\xxx,\xxx^-) \asymp c(\xxx, \yyy) \asymp c(\yyy,\yyy^-)$ for $\xxx \sim_h \yyy $. It follows that $p_\xxx \asymp p_\yyy$ for all $\xxx \sim_h \yyy$, which is a strong restriction on the choice of possible probability weights $\{p_i\}_{i=1}^N$.  For the $\lambda$-NRW this restriction is fulfilled,  as the natural weights satisfy $p_\xxx=r_\xxx^\alpha \asymp r^{\alpha|\xxx|}$ for all $\xxx \in X$. In particular, for a homogeneous IFS that satisfies the OSC and has contraction ratio $r$,  the SRW on a pre-augmented tree $(X, {\mathfrak E})$ is a $\frac 1N$-NRW, as  $p_i = 1/N = r^\alpha$, $\lambda = N^{-1}$ and $c(\xxx, \xxx^-)=1$.

\vspace{1mm}

However, for the $\lambda$-QNRW, we have to choose special probability weights so that $p_\xxx \asymp p_\yyy$, $\xxx \sim_h \yyy$. For instance, if $K$ is the $d$-dimensional Sierpi\'{n}ski graph, one can derive from $p_\xxx \asymp p_\yyy$, $\xxx \sim_h \yyy$ that all $p_i$'s are equal. Hence the $\lambda$-QNRW on the corresponding augmented tree must be the $\lambda$-NRW.

\medskip

\begin{example} {\rm If $K$ is the Sierpi\'{n}ski carpet, and the IFS is $\{S_i\}_{i=1}^8$ on $\mathbb{R}^2$, where $S_i(z)=(z+q_i)/2$, $i=1,2,\ldots,8$, and the $q_i$'s are  the four vertices and four mid-points of the edges of a square, and is labeled clockwise starting from the top left corner. Consider the pre-augmented tree $(X, {\mathfrak E})$ with ${\mathfrak E}_h$ defined by $\xxx \sim_h\yyy$ if $ |\xxx| = |\yyy|$ and ${\rm dim_H}\big (K_\xxx \cap K_\yyy \big ) =1$.  Then it is direct to check that one obtains a $\lambda$-QNRW with probability weight $\{p_i\}_{i=1}^8$ on $(X, {\mathfrak E})$ if and only if $p_1=p_3=p_5=p_7$, $p_2=p_6$, and $p_4=p_8$. Note that with this weight, the self-similar measure is a {\it doubling measure }\cite{Y}, i.e, there exists a constant $C>0$ such that for any $\xi \in K$ and any $\delta > 0$, we have $\mu(B(\xi;2\delta)) \leq C \mu(B(\xi;\delta))$. }
\end{example}

 \medskip

 In general, we can make use of a result in \cite{Y} to characterize the probability weights $\{p_i\}_{i=1}^N$ that admit a $\lambda$-QNRW.

\medskip

\begin{theorem} \label{th4.4'}
Let $\{S_i\}_{i=1}^N$ be an IFS with OSC, $K$ be the self-similar set, and $\mu$ be the self-similar measure generated by $\{p_i\}_{i=1}^N$.  Then $\{p_i\}_{i=1}^N$ admits a $\lambda$-QNRW on the augmented tree $(X,\ed)$ if and only if $\mu$ is a doubling measure on $K$.
\end{theorem}

\begin{proof} Let $K_\xxx = S_\xxx(K)$. For  $F\subset {\mathbb R}^d$, let $|F|= \diam (F)$, and $B(F;\delta):=\{\xi \in \mathbb{R}^d : \ {\rm dist}(\xi,F) \leq \delta\}$ be the closed $\delta$-neighborhood of $F$. Consider the following two conditions:

\vspace {0.1cm}

(i) if there exist constants $C_1$, $C_2>0$ such that for any  $\www$, $\vvv \in \Sigma^*\setminus \{\vartheta\}$ that satisfy $K_\www \subset B(K_\vvv;\ C_1 r_\vvv)$, then $p_\www \leq C_2 p_\vvv$;

\vspace {0.1cm}

(ii) Replace the quantifiers of (i) to ``if for any $C_1 >0$, there exists $C_2>0$ such that $\ldots$".

\vspace{0.1cm}

\noindent In \cite[Theorem 2.3]{Y}, it was proved that (i) $\Rightarrow $ $\mu$ is a doubling measure $\Rightarrow $ (ii). We will prove  (ii) $\Rightarrow $  $\{p_i\}_{i=1}^N$ admits a QNRW $\Rightarrow $ (i). Hence all four conditions are equivalent, proving the theorem.

\vspace{0.1cm}

Assume (ii) and let $C_1=(\gamma+|K|)r^{-1}$, where $r=\min_{1\leq i\leq N}r_i$, and $\gamma$ is as in \eqref{eq3.3}. Then by \eqref{eq3.3}, $\xxx \sim_h \yyy$ implies
\begin{align*}
K_\xxx \subset &B(K_\yyy;\ \gamma  r^{|\xxx|}+r_\xxx|K|) \\
\subset & B(K_\yyy;\ (\gamma+|K|)r^{|\xxx|}) \subset B(K_\yyy;\ C_1 r_\yyy).
\end{align*}
Then by (ii), $p_\xxx \leq C_2 p_\yyy$ for some $C_2>0$; by the same reason, $p_\yyy \leq C_2 p_\xxx$. Hence for $0< \lambda <1$, $c(\xxx,\xxx^-)=p_\xxx \lambda^{-|\xxx|} \asymp p_\yyy \lambda^{-|\yyy|} = c(\yyy,\yyy^-)$ for all $\xxx \sim_h \yyy$. We can choose, for instance, $c(\xxx,\yyy) = \sqrt{c(\xxx,\xxx^-)c(\yyy,\yyy^-)}$ to get a $\lambda$-QNRW on $(X,\ed)$.

\vspace {0.2cm}

Next we assume  $\{p_i\}_{i=1}^N$ admits a $\lambda$-QNRW on $(X,\ed)$.
We first define an integer $N_0 =\lceil\log r/\log r'\rceil$ where $r'=\max_{1 \leq i\leq N}r_i$. Then for $p=\min_{1\leq i\leq N} p_i$, and for any $\zzz \in X$, we have $p_\zzz \geq p^{N_0}p_{\zzz^-}$ since the difference of the length of the words  $\zzz$ and $\zzz^-$ is at most $N_0$.

\vspace {0.1cm}
The $\lambda$-QNRW implies that $
p_\xxx \lambda^{-|\xxx|} = c(\xxx,\xxx^-) \asymp c(\xxx,\yyy) \asymp c(\yyy,\yyy^-) = p_\yyy \lambda^{-|\yyy|}$ for any $\xxx \sim_h \yyy$.
 Therefore there exists $C>0$ such that $p_\xxx \leq Cp_\yyy$ for any $\xxx \sim_h \yyy$. Let $C_1 = \min\{\gamma, (1-r)(2r)^{-1}|K|\}>0$. It follows that  for any $\yyy \in \mathcal J_m$,
\begin{align*}
|B(K_\yyy;\ C_1r_\yyy)| &\leq |K_\yyy|+2C_1r_\yyy \\
&\leq (1+2(1-r)(2r)^{-1})r^m|K|=r^{m-1}|K|.
\end{align*}
Hence for $\xxx \in X$ such that $K_\xxx \subset B(K_\yyy;\ C_1r_\yyy)$, $|\xxx| \geq m-1$ (by \eqref{eq3.1''}).  We claim that $p_\xxx \leq Cp^{-N_0}p_\yyy$. Indeed, let $\xxx_m \in \mathcal J_m$ such that it is on the same vertical path of $\xxx$. Then either $K_\xxx \subset K_{\xxx_m}$ or $K_{\xxx_m} \subset K_\xxx$. We have three distinct cases:

\vspace{0.1cm}

\noindent \hspace {0.23cm}  (a) If $\xxx_m=\yyy$, then $p_\xxx \leq p_{\yyy^-} \leq p^{-N_0}p_\yyy$.

\vspace{0.1cm}

\noindent \hspace {0.2cm}  (b) If $\xxx_m \neq \yyy$ and $|\xxx|=m-1$, then $\xxx_m^-=\xxx$ and $K_{\xxx_m} \subset K_\xxx \subset {B}(K_\yyy;\ C_1r_\yyy)$, which implies ${\rm dist} (K_{\xxx_m},K_\yyy) \leq C_1r_\yyy \leq \gamma r^m$, hence $\xxx_m \sim_h \yyy$. This shows that $p_\xxx \leq p^{-N_0}p_{\xxx_m} \leq Cp^{-N_0}p_\yyy$.

\vspace{0.1cm}

 \noindent \hspace {0.23cm} (c) If $\xxx_m \neq \yyy$ and $|\xxx|\geq m$, then $K_{\xxx} \subset K_{\xxx_m}$ and ${\rm dist}(K_{\xxx_m},K_\yyy) \leq {\rm dist} (K_\xxx,K_\yyy) \leq C_1r_\yyy \leq \gamma r^m$, hence $\xxx_m \sim_h \yyy$. This shows that $p_\xxx \leq p_{\xxx_m} \leq Cp_\yyy$.

\vspace{0.1cm}

 To  conclude the proof,  we let $\www$, $\vvv \in \Sigma^*\setminus \{\vartheta\}$. We can choose $\uuu$, $\ttt \in X$ such that $S_\uuu(K) \subset K_\www \subset K_{\uuu^-}$, and $K_\ttt \subset K_\vvv \subset K_{\ttt^-}$. If $K_\www \subset B(K_\vvv;\ C_1r_\vvv)$ as in the assumption in condition (i), then $K_\uuu \subset B(K_{\ttt^-};\ C_1r_{\ttt^-})$, hence
\begin{equation*}
p_\www \leq p_{\uuu^-} \leq p^{-N_0}p_\uuu \leq Cp^{-2N_0}p_{\ttt^-} \leq Cp^{-3N_0}p_\ttt \leq Cp^{-3N_0}p_\vvv.
\end{equation*}
Let $C_2=Cp^{-3N_0}$. Then (i) follows.
\end{proof}

\bigskip

\section{Martin boundary and hitting distribution}
\label{sec:5}

\noindent In this section we focus on the boundary behavior of natural random walks on augmented trees. We first show that Ancona's Theorem (Theorem \ref{th2.5}) can be applied to identify the Martin and hyperbolic boundaries with the self-similar sets. This extends \cite[Theorem 4.7]{Ka} for the SRW on the Sierpi\'{n}ski graph.

\vspace{1mm}

\begin{theorem} \label{th5.1}
Let  $(X,{\mathfrak E})$ be a pre-augmented tree which is hyperbolic and has bounded degree, and let  $\{Z_n\}_{n=0}^\infty$ be a $\lambda$-QNRW on $(X,{\mathfrak E})$. Then the transition probability $P$ satisfies the hypotheses in Theorem \ref{th2.5}. Hence  ${\mathcal M} = {\mathcal M}_{\min}$, and  the hyperbolic boundary $\partial_H X$  and ${\mathcal M}$ are homeomorphic under the canonical mapping.

In particular, if the IFS satisfies the OSC, and $(X,{\mathfrak E})$ is the augmented tree as in \eqref{eq3.3}, then $\partial_H X$, $\mathcal{M} $  and the self-similar set $K$ are all homeomorphic under the canonical mapping.
\end{theorem}

\begin{proof}
We need to check that the conditions in Ancona's theorem  (Theorem \ref{th2.5}) are satisfied. Clearly $P$ is of bounded range. That
\begin{equation} \label{eq5.0}
P(\xxx,\yyy) = \frac{c(\xxx,\yyy)}{m(\xxx)} \geq \eta >0, \qquad \xxx \sim \yyy
\end{equation}
follows from Definition \ref {th4.4} and \eqref{eq4.6}. Hence $P$ is uniformly irreducible. It remains to show that $r(P)<1$. By Proposition \ref{th2.7}, it suffices to show that $(X,P)$ satisfies the strong isoperimetric inequality.

\vspace{0.1cm}

Consider the subtree $T = (X,{\mathfrak E}_v)$, and restrict the random walk to $T$. The transition probability is given by  $P_T(\xxx, \xxx^-) = \frac {c(\xxx, \xxx^-)}{m_T (\xxx)}$, and $=0$ otherwise, where  $m_T(\xxx) = c(\xxx,\xxx^-)+\sum_{\yyy:\yyy^-=\xxx} c(\xxx,\yyy)$ for $\xxx \in X \setminus \{\vartheta\}$. Note that $m_T(\vartheta) = m(\vartheta)$. We claim that the transition probability  $P_T$ satisfies the following strong isoperimetric inequality:
\begin{equation} \label{eq5.1}
m_T(A) \leq \frac{1+\lambda}{1-\lambda}c_T(\partial A)
\end{equation}
for all finite subsets $A \subset X$.
To prove the claim, we can first assume that $A$ is nonempty and connected (otherwise check \eqref{eq5.1} for each connected component and then sum them up). We use induction on $n:=\min\{m: A \subset X_m\}$. It is clear that if $n=0$, then $A=\{\vartheta\}$ and $m_T(A) = m(\vartheta) = c_T(\partial A)$ implies \eqref{eq5.1}. Suppose that for some $n \geq 1$, \eqref{eq5.1} holds for any finite connected subset in $ X_{n-1}$. Consider a connected subset $A \subset X$ such that $\min\{m: A \subset X_m\} = n$. If $\# A = 1$, then $m_T(A) = c_T(\partial A)$ implies \eqref{eq5.1}. Now we suppose $\# A \geq 2$. Then for any $\xxx \in A \cap \mathcal J_n$, $\xxx^- \in A$ since $T$ is a tree and $A$ is connected. Hence
\begin{align*}
m_T(A)
%&=  m_T(A \setminus \mathcal J_n) + m_T(A \cap \mathcal J_n) \\
&= m_T(A \setminus \mathcal J_n) + \sum_{\xxx \in A\cap \mathcal J_n} \Big( c(\xxx,\xxx^-) + \sum_{\yyy: \yyy^-=\xxx} c(\xxx,\yyy) \Big) \\
&= m_T(A \setminus \mathcal J_n) + \sum_{\xxx \in A\cap \mathcal J_n} (1+\lambda^{-1})c(\xxx,\xxx^-) \qquad \qquad  \hbox{(by \eqref{eq4.1})} \\
&\leq \frac{1+\lambda}{1-\lambda}\Big(c_T(\partial(A \setminus \mathcal J_n))+\sum_{\xxx \in A\cap \mathcal J_n} (\lambda^{-1}-1)c(\xxx,\xxx^-)\Big) \quad \hbox{(by induction)}\\
&= \frac{1+\lambda}{1-\lambda}\Big(c_T(\partial(A \setminus \mathcal J_n))+\sum_{\xxx \in A\cap \mathcal J_n}\big(\sum_{\yyy: \yyy^-=\xxx} c(\xxx,\yyy)-c(\xxx,\xxx^-)\big)\Big) \quad \hbox{(by \eqref{eq4.1})}\\
& = \frac{1+\lambda}{1-\lambda} c_T(\partial A)\qquad  \qquad \hbox {(use tree property)}
\end{align*}
This completes the proof of the claim.  Next observe that for any $\xxx \in X \setminus \{\vartheta\}$, there exists $C_1>0$ such that
\begin{align*}
m(\xxx) &\leq C_1p_\xxx\lambda^{-|\xxx|} = C_1c(\xxx,\xxx^-) \qquad \qquad \quad \hbox{(by bounded degree and \eqref{eq4.6})}\\
&= \frac{C_1\lambda}{1+\lambda}\big(\sum_{\yyy: \yyy^-=\xxx} c(\xxx,\yyy)+c(\xxx,\xxx^-)\big) \quad ~~\hbox{(by \eqref{eq4.1})} \\
&= \frac{C_1\lambda}{1+\lambda}m_T(\xxx).
\end{align*}
This together with the above claim imply the following strong isoperimetric inequality on $(X,{\mathfrak E})$:
$$
m(A) \leq \frac{C_1\lambda}{1+\lambda}m_T(A) \leq \frac{C_1\lambda}{1-\lambda}c_T(\partial A) \leq \frac{C_1\lambda}{1-\lambda}c(\partial A), \quad A \subset X\ \hbox{finite}.
$$

For the last part, we observe the OSC implies that the the augmented tree $(X, {\mathcal E})$ is of bounded degree (Proposition \ref{th3.7}), and Theorem \ref {th3.6} and the above give the homeomorphism among $K$, $\partial_HX$ and ${\mathcal M}$.
\end{proof}

\bigskip

Since $r(P)<1$ as shown in the above proof, we can apply the first part of Theorem \ref{th2.5} to obtain the same inequality for $F(\xxx, \yyy)$ of the $\lambda$-QNRW.

\medskip

\begin{corollary} \label{th5.2}
Let $(X,{\mathfrak E})$ be a hyperbolic pre-augmented tree with bounded degree, and let  $\{Z_n\}_{n=0}^\infty$ be a $\lambda$-QNRW on $(X,{\mathfrak E})$. Then for $\delta \geq 0$, there is a constant $C_\delta \geq 1$ such that for any $\xxx$, $\yyy \in X$ and $\mathbf{u}$  within distance $\delta$ from a geodesic segment between $\xxx$ and $\yyy$,
\begin{equation}  \label{eq5.2}
F(\xxx, \mathbf{u})F(\mathbf{u}, \yyy) \leq
F(\xxx, \yyy) \leq C_\delta  F(\xxx, \mathbf{u})F(\mathbf{u}, \yyy).
\end{equation}
\end{corollary}

\bigskip

The above corollary allows us to give a useful estimate  for $F(\xxx, \yyy)$ in terms of the Gromov product, which will be essential in the estimates of the Martin kernel and the Na\"{i}m kernel.  For this,  we need to restrict our consideration to the IFS satisfying the OSC, and to the class of natural random walks with natural weights $p_i = r_i^\alpha, 1\leq i \leq N$. In this case,  we have  $p_{\xxx} = \mu (S_{\xxx}(K)) = r_{\xxx}^\alpha \asymp r^{\alpha m}$ for $\xxx \in {\mathcal J}_m$, where $r= \min_{1\leq i\leq N}\{r_i\}$, $\alpha$ is the Hausdorff dimension of $K$, and  $\mu$ is the self-similar measure corresponding to $\{p_i\}_{i=1}^N$ (Lemma \ref{th3.0}).

\medskip

\begin{theorem} \label{th5.3}
Let $\{S_i\}_{i=1}^N$ be an IFS satisfying the OSC, and let $\{Z_n\}_{n=0}^\infty$ be a $\lambda$-NRW on the augmented tree $(X, {\mathfrak E})$. Then
$$
F(\xxx, \yyy) \asymp \lambda^{|\xxx|-(\xxx|\yyy)}r^{\alpha(|\yyy|-(\xxx|\yyy))}, \qquad  \forall \ \xxx, \yyy \in X
$$
where $(\xxx|\yyy)$ is the Gromov product. Moreover, the Martin kernel satisfies
$$
K(\xxx,\yyy)  \asymp \lambda^{|\xxx|-(\xxx|\yyy)}r^{-\alpha(\xxx|\yyy)}, \qquad  \xxx, \yyy \in X,
$$
and also for $K(\xxx, \xi)$ by replacing $\yyy$ with $\xi \in \partial_HX  (\approx {\mathcal M}\approx K)$ in the above estimate.
\end{theorem}

\begin{proof}
Consider the canonical geodesic $\pi(\xxx, \yyy) = \pi(\xxx, \uuu) \cup \pi(\uuu,\vvv) \cup \pi(\vvv,\yyy)$ as in Section~\ref{sec:3}.
It follows from Corollary \ref{th5.2} that
$$
F(\xxx,\yyy) \asymp F(\xxx,\uuu)F(\uuu,\vvv)F(\vvv,\yyy)
\asymp \dfrac{F(\xxx, \vartheta)}{F(\uuu, \vartheta)} \cdot
F(\uuu, \vvv)\cdot \dfrac{F(\vartheta,\yyy)}{F(\vartheta,\vvv)}.
$$
Since $d(\uuu,\vvv) \leq M$ (Proposition \ref{th3.2}), \eqref{eq5.0} yields $\eta^M \leq F(\uuu,\vvv) \leq 1$.
Applying Theorem \ref{th4.1}  ($F(\www, \vartheta) = \lambda^{|\www|}$), Theorem \ref{th4.6} ($F(\vartheta, \www)  \asymp r^{\alpha |\www|}$) and $d(\uuu,\vvv) \leq M$ to the other two factors of the above expression, we have
$$
F(\xxx,\yyy) \asymp \lambda^{d(\xxx, \uuu)}r^{\alpha d(\vvv, \yyy)} \asymp \lambda^{d(\xxx, \uuu)}r^{\alpha d(\vvv, \yyy)} (\lambda r^\alpha)^{d(\uuu,\vvv)/2}= \lambda^{|\xxx|-(\xxx|\yyy)}r^{\alpha(|\yyy|-(\xxx|\yyy))}.
$$
\indent For the estimate of the  Martin kernel, we recall that $K(\xxx,\yyy)= \frac {G(\xxx, \yyy)}{G(\vartheta,\yyy)} = \frac{F(\xxx,\yyy)}{F(\vartheta,\yyy)}$, and the Gromov product $(\xxx|\xi)$ is defined for $\xi \in \partial_HX$  (see (\ref {eq3.4})). Using the above estimate of $F(\xxx,\yyy)$, the second part follows.
\end{proof}

\medskip

In the following, we will assume the IFS satisfies the OSC, so the augmented tree $(X, {\mathfrak E})$ has bounded degree (Proposition \ref {th3.7}). By Theorem \ref{th5.1}, we identify $K$ with the Martin boundary ${\mathcal M}$ of a $\lambda$-QNRW $\{Z_n\}_{n=0}^{\infty}$.  We will study the hitting distribution on ${\mathcal M} \approx K$ starting from the root $\vartheta$.

\medskip

First we construct a projection $\{Z_n\}_{n=0}^\infty$ on $X$ onto $K$. For $\{S_i\}_{i=1}^N$ satisfying the OSC, we let $O$ be an open set in the OSC such that $O\cap K \not = \emptyset$ \cite{Sc}. Define a projection $\iota: X \to K $ by selecting arbitrarily
\begin{equation*}
\iota (\xxx) \in S_{\xxx} (O \cap K), \quad \xxx \in X.
\end{equation*}
We extend $\iota$ to $\widehat X = X \cup {\mathcal M} $ by defining $\iota(\xi) = \xi$ for $\xi \in {\mathcal M}  (\approx K)$. Then a sequence $\{\xxx_n\}_n \subset X$ with $|\xxx_n| \to \infty$ converges to
$\xi \in \widehat X$ if and only if $\iota(\xxx_n) \to \iota(\xi) \in K$ in the Euclidean topology. Hence $\iota$ is continuous on $\widehat{X}$, and $\iota |_{\mathcal M}$ is the canonical homeomorphism of  ${\mathcal M}$ onto $K$. Also we have
\begin{equation}\label {eq5.3}
\iota (Z_{\infty}) = \lim\limits_{n \to \infty} \iota (Z_n) \quad {\mathbb P_\vartheta} -a. e.
\end{equation}
Let $\nu_\vartheta$ be the distribution of $Z_\infty$, and  $\nu$  the distribution of $\iota (Z_\infty)$. It is direct to check that $\nu = \nu_{\vartheta}\circ \iota^{-1}$ on $K$.  For later use, we shall write $\mathbb P = \mathbb P_\vartheta$ and $\mathbb E = \mathbb E_\vartheta$ if there is no confusion.

\bigskip

For $\ell \geq 1$, let $\tau_\ell : = \inf \{n \geq 0: Z_n \in \mathcal{J}_\ell\}$ be the first hitting time of $\mathcal{J}_\ell$. Then by
Theorem \ref{th4.6}, the distribution of $Z_{\tau_\ell}$ is given by
\begin{equation} \label{eq5.4}
\mathbb{P}(Z_{\tau_\ell} = \xxx )= p_\xxx = \mu(S_\xxx(K)), \quad \xxx \in \mathcal{J}_\ell,
\end{equation}
where $\mu$ is the self-similar measure associated with the probability weight $\{p_i\}_{i=1}^N$. Note that under the OSC, $\mu (S_\xxx(K) \cap S_\yyy(K))=0$ for $\xxx \in {\mathcal J}_m, \ \xxx \neq \yyy$.

\bigskip

\begin{proposition} \label{th5.4}
For $\xxx \in \mathcal{J}_m$ and $\ell \geq m $,
\begin{equation} \label{eq5.5}
\mathbb{P}\big (\iota(Z_{\tau_\ell} ) \in S_{\xxx} (K)\big ) = \mu (S_\xxx(K)).
\end{equation}
Also  $\mathbb{P}\big (\iota(Z_{\tau_\ell} ) \in S_{\xxx} (K)\cap S_{\yyy} (K)\big ) = 0$ for any $\xxx, \yyy \in {\mathcal J}_m, \ \xxx \not = \yyy$.
\end{proposition}

\begin{proof}
Let $\xxx, \zzz \in \mathcal{J}_m$, $\zzz\neq \xxx$. Since $\iota(\zzz) \in S_{\zzz} (O \cap K)$, the OSC implies that $S_\zzz(O) \cap S_\xxx(O) =\emptyset$. Hence $\iota(\zzz) \not \in \overline {S_\xxx(O)}\supset {S_\xxx(K)}$. It follows that  for $\xxx \in \mathcal{J}_m$   and $\iota(Z_{\tau_\ell}(\omega)) \in S_{\xxx}(K)$, $\ell \geq m$, $Z_{\tau_\ell}(\omega)$ must be of the form $\xxx\uuu$ for some $\uuu \in \Sigma^*$ and $\xxx\uuu \in {\mathcal J}_\ell$.
By \eqref{eq5.4},
\begin{align*}
\mathbb{P}(\iota(Z_{\tau_\ell}) \in S_{\xxx} (K))
& = \sum_{\uuu \in \Sigma^\ast: \xxx\uuu \in \mathcal{J}_\ell}\mathbb{P}(Z_{\tau_\ell} = \xxx  \uuu) \\
& = \sum_{\uuu \in \Sigma^\ast: \xxx\uuu \in \mathcal{J}_\ell}\mu(S_{\xxx\uuu}(K))
= \mu(S_\xxx(K)) .
\end{align*}
Using the fact that $\mathbb{P}(Z_{\tau_\ell} = \xxx  \zzz, \ Z_{\tau_\ell}= \yyy\zzz')= 0 $ and adopting the same argument, the second part follows.
\end{proof}

\medskip

To obtain the hitting distribution $\nu_\vartheta$, a simple minded approach is to take limit of $\ell$ in \eqref{eq5.5} to obtain $\mathbb{P}\big (\iota (Z_{\infty} )\in S_{\xxx} (K)\big ) = \mu(S_\xxx(K))$. However this requires that $\chi_{S_{\xxx} (K)} (\iota (Z_{\tau_\ell})) \to \chi_{S_{\xxx} (K)} (\iota (Z_\infty))$ a.e., but we cannot make such conclusion directly, as the convergence for the composition $f(\iota (Z_{\tau_\ell}))$ generally requires the continuity of $f$ \cite{BJ}. We need to go through a more detailed analysis.

\medskip

\begin{lemma} \label{th5.5} For any two distinct $\xxx, \xxx'\in {\mathcal J}_m$, we have
$$
\nu (S_\xxx(K) \cap S_{\xxx'} (K)) =0.
$$
\end{lemma}

\begin {proof}  For $\xxx \in {\mathcal J}_m$, and $k>m$, we let
$$
{\mathcal F}_{k, \xxx} = \{\yyy\uuu \in {\mathcal J}_k:\ \yyy \in {\mathcal J}_m \setminus \{\xxx\},\  \uuu \in \Sigma^*,  \ S_{\yyy\uuu}(K) \cap S_{\xxx}(K)\neq \emptyset\}
$$
Let $F_{k,\xxx} = \bigcup_{\www \in {\mathcal F}_{k, \xxx} }S_{\www}(K)$. Then $F_{k,\xxx}$ is a decreasing sequence of sets in $k$ and the limit set is $\bigcup_{\yyy\in {\mathcal J}_m\setminus \{\xxx\}} (S_\yyy(K) \cap S_\xxx(K))$, which is a $\mu$-null set by the OSC.  This implies that $ \lim_{k\to \infty} \mu (F_{k,\xxx}) =0$.

\vspace {0.1cm}
Let  $T_{k, \xxx} = S_{\xxx}(K) \bigcup F_{k, \xxx}$. By a similar argument as in Proposition \ref {th5.4},  we can show that
\begin{equation} \label {eq5.6}
\mathbb{P}\big (\iota(Z_{\tau_\ell} ) \in T_{k,\xxx} \cap T_{k, \yyy} \big ) =\mu (T_{k,\xxx} \cap T_{k, \yyy})\qquad  \hbox {for} \ \ell \geq m.
\end{equation}
Observe that $K \setminus T_{k, \xxx}$ has positive Euclidian distance to $S_{\xxx}(K)$. Hence $\iota (Z_{\infty}) \in S_{\xxx}(K)$ implies that $\iota (Z_n)$ is eventually in  ${T_{k,\xxx}}$, which means $\lim_{n \to \infty}
\chi_{T_{k, \xxx}}(\iota(Z_n)) = 1$.

\vspace {0.1cm}
 Now for two distinct $\xxx, \xxx' \in {\mathcal J}_m$, if $S_\xxx (K) \cap S_{\xxx'} (K) = \emptyset$, then the lemma is obviously true. Hence assume that $S_\xxx (K) \cap S_{\xxx'} (K) \neq \emptyset$. Let $T_{k, \xxx}$ and  $T_{k, {\xxx'}}$ be defined as the above. Then
\begin{align} \label {eq5.7}
\nu(S_{\xxx}(K) \cap S_{\xxx'}(K))
&= \mathbb{P} (\iota (Z_{\infty}) \in S_{\xxx}(K)\cap S_{\xxx'}(K))
\nonumber \\
&\leq \mathbb{E} \left(\liminf_{n \to \infty}~\chi_{T_{k, \xxx} \cap T_{k, \xxx'} }(\iota(Z_n)) \right)\nonumber\\
&\leq \liminf_{n \to \infty} ~\mathbb{E} \big( \chi_{T_{k, \xxx} \cap T_{k, \xxx'}}(\iota(Z_n)) \big)  \qquad \ \hbox {(by Fatou's lemma)}  \nonumber \\
%&\leq \liminf_{\ell \to \infty} ~\mathbb{E} \big ( \chi_{T_{k, \xxx}\cap T_{k, \xxx'} }(\iota(Z_{\tau_\ell})) \big) \qquad  \hbox {(as subsequence)} \\
&\leq \liminf_{\ell\to \infty} ~\mathbb{P} (\iota(Z_{\tau_\ell}) \in T_{k, \xxx}\cap T_{k, \xxx'}) \\
& = \mu(T_{k, \xxx}\cap T_{k, \xxx'})  \qquad \qquad \qquad
 \quad ~~~\hbox {(by \eqref{eq5.6})}.  \nonumber
\end{align}
By the limit in the first paragraph, we have  $\nu(S_{\xxx}(K)\cap S_{\xxx'}(K) ) \leq \lim_{k\to \infty} \mu(T_{k, \xxx}\cap T_{k, \xxx'})=
 \mu((S_{\xxx}(K)\cap S_{\xxx'}(K)) = 0$, and this completes the proof.
\end{proof}

\medskip

\begin{theorem} \label{th5.6}
Let $\{S_i\}_{i=1}^N$ be an IFS with the OSC, and let $K$ be the self-similar set.
Suppose $\{Z_n\}_{n=0}^\infty$ is a $\lambda$-QNRW on the augmented tree $(X, \mathfrak E)$ associated with a set of probability weights $\{p_i\}_{i=1}^N$. Then the hitting distribution $\nu_{\vartheta}$ is the self-similar measure $\mu$ with weights $\{p_i\}_{i=1}^N$ on $K$.
\end{theorem}

\medskip

\begin{proof}
To prove the theorem, it suffices to show that
$\nu(S_{\xxx} (K)) = p_\xxx$ for $m \geq 1$ and $\xxx \in \mathcal J_m$.
Then $\nu$ is the unique self-similar measure with weights $\{p_i\}_{i=1}^N$.
To this end, we fix $m \geq 1$ and $\xxx \in \mathcal J_m$.
Let
$$
U_{\xxx} = S_{\xxx}(K) \setminus \bigcup_{\yyy \in \mathcal J_m: \yyy \neq \xxx} S_{\yyy}(K)
$$
be the `interior' of the $m$-cell $S_{\xxx}(K)$. Then any sequence of points in $K$ that has a limit in $U_{\xxx}$ must be in $U_{\xxx}$ eventually. Therefore $\iota (Z_{\infty}) \in U_{\xxx}$ implies
$\lim_{n \to \infty} \chi_{U_{\xxx}} (\iota(Z_n)) = 1.$
Hence by using Fatou's lemma as in Lemma \ref{th5.4}, we have
\begin{align*}
\nu(S_{\xxx}(K)) & = \nu(U_{\xxx}) \qquad  \qquad \qquad \qquad \quad \hbox {(by Lemma \ref{th5.5})}\\
%&= \mathbb{P} (\iota (Z_{\infty}) \in U_{\xxx}(K)) \\
&\leq \liminf_{\ell \to \infty} ~\mathbb{P} (\iota(Z_{\tau_\ell}) \in U_{\xxx}) \quad ~~\hbox {(as in \eqref{eq5.7})} \\
&\leq \liminf_{\ell \to \infty} ~\mathbb{P} (\iota(Z_{\tau_\ell}) \in S_{\xxx}(K)) \\
&= \mu(S_\xxx(K)) = p_\xxx  \qquad \qquad \quad \hbox {(by Proposition \ref{th5.4})}.
\end{align*}
 Summing the above inequality over $\xxx \in \mathcal J_m$, we have $1 = \nu(K) \leq \sum_{\xxx \in \mathcal J_m} \nu(S_\xxx(K)) \leq \sum_{\xxx \in \mathcal J_m} p_\xxx = 1$. Thus $\nu(S_{\xxx} (K)) = p_\xxx$ for $\xxx \in \mathcal J_m$. This completes the proof of the theorem.
\end{proof}

\bigskip

\section{Estimation of the Na\"{i}m kernel}
\label{sec:6}

Let $P$ be a transient reversible random walk on a countable set $X$.
For $\xxx,\yyy \in X$, we define the $\Theta$-kernel by
\begin{equation*}
\Theta(\xxx,\yyy) = \dfrac{K(\xxx,\yyy)}{G(\xxx,\vartheta)} = \dfrac{F(\xxx,\yyy)}{F(\xxx,\vartheta)G(\vartheta,\vartheta)F(\vartheta,\yyy)}.
\end{equation*}
This kernel is clearly positive and symmetric (by $m(\xxx)G(\xxx,\yyy)=m(\yyy)G(\yyy,\xxx)$) on $X \times X$, and can be extended to  $\Theta (\xxx, \eta)$ on  $X \times \widehat{X}$ continuously (as $K(\xxx, \yyy)$ does), but it is more difficult to extend to $\widehat{X} \times \widehat{X}$. We first recall the extension of $\Theta (\xi, \eta)$ on $\widehat{X} \times \widehat{X}$ in \cite{Si}. For $\xi \in \mathcal{M}$, we define the {\it $\xi$-process}  by setting the transition probability
\begin{equation} \label{eq6.1}
P^{\xi}(\xxx,\yyy) = \dfrac{P(\xxx,\yyy)K(\yyy,\xi)}{K(\xxx,\xi)}.
\end{equation}
(It is also known as the {\it $h$-transform} of $P$ \cite{Dy}, by taking the harmonic function $h = K(\cdot, \xi)$ on $X$.) Clearly $P^{\xi}$ is a transition probability since $K(\cdot, \xi)$ is harmonic on $X$. For the $\xi$-process, we denote the corresponding ever-visiting probability and Green function by $F^\xi(\cdot,\cdot)$ and $G^\xi(\cdot,\cdot)$ respectively.
The following lemma can be found in \cite[Corollary 7.51]{Wo2}.

\medskip

\begin{lemma} \label{th6.1}
For $\xi \in \mathcal{M}$, let $\nu_\vartheta^\xi$ be the hitting distribution of the $\xi$-process. Then $\xi \in \mathcal{M}_{\min}$ if and only if
$\nu_\vartheta^\xi = \delta_{\xi}$, the point mass at $\xi$.
\end{lemma}

For $m \geq 0$, let $\tau_m^*$ be the last visit time of $X_m = \bigcup_{k=0}^m \mathcal{J}_k$ by $\{Z_n\}_{n=0}^\infty$.  Note that $\tau_m^*$ is not a stopping time. For $\xi \in K$ and $\zzz \in X$, let
\begin{equation*}
\ell_m^\xi(\zzz) = \mathbb{P}_{\vartheta}^\xi (Z_{\tau_m^*}=\zzz),
\end{equation*}
where $\mathbb{P}_{\vartheta}^\xi$ is the probability corresponding to the $\xi$-process. The following lemma is from \cite[Lemma 1.9]{Si}.

\medskip

\begin{lemma} \label{th6.2} Let $\xi, \eta \in  \mathcal{M}$. Then the sum
$
\sum\limits_{\zzz \in X} \ell_m^\xi(\zzz) \Theta(\zzz,\eta)
$
is increasing in $m$.
\end{lemma}

In view of Lemmas \ref{th6.1} and \ref{th6.2}, Silverstein \cite{Si} extended the Na\"{i}m kernel on ${\mathcal M} \times {\mathcal M}$  by
$\Theta(\xi, \eta) = \lim_{k \to \infty} {\sum}_{\zzz \in X}~\ell_k^\xi(\zzz) \Theta(\zzz, \eta)$ for $\xi, \eta \in \mathcal{M}$.
We will apply this to the $\lambda$-NRW on the augmented tree. Since the walk satisfies $|Z_{n+1}|-|Z_n| \in \{0,\pm 1\}$, the identity  reduces to
\begin{equation} \label{eq6.3}
 \Theta(\xi, \eta) = \lim_{k \to \infty} {\sum}_{\zzz \in {\mathcal J}_k}\ell_k^\xi(\zzz) \Theta(\zzz, \eta).
\end{equation}
Here is the main theorem of this section.

\medskip

\begin{theorem} \label{th6.4}
Let $\{S_i\}_{i=1}^N$ be an IFS that satisfies the OSC, and let $K$ be the self-similar set. Let $\{Z_n\}_{n=0}^{\infty}$ be a $\lambda$-NRW  on the augmented tree $(X, {\mathfrak E})$.  Then
\begin{align*}
 \Theta(\xi, \eta) \asymp (\lambda r^\alpha)^{-(\xi|\eta)}, \quad  \xi, \eta \in \partial_H{X}, \ \xi \neq \eta,
\end{align*}
where $\alpha = \dim_HK$. Consequently, by Theorem \ref{th3.6}, we have
$$
\Theta(\xi, \eta) \asymp |\xi - \eta|^{-(\alpha+\beta)}, \quad \xi, \eta \in K, \ \xi \neq \eta,
$$
with $\beta = \frac{\log \lambda}{\log r}$.
\end{theorem}

\begin{proof}
Let $\xi,\eta \in \partial_HX$, $\xi\neq \eta$. For $\xxx \in X$, Proposition \ref{th4.1} and Theorem \ref{th5.3} imply
\begin{align}\label{eq6.4}
\Theta(\xxx,\eta) =\frac{K(\xxx,\eta)}{F(\xxx,\vartheta)G(\vartheta,\vartheta)} \asymp (\lambda r^\alpha)^{-(\xxx|\eta)}.
\end{align}
We fix two geodesic rays $(\xxx_n)_{n=0}^\infty$ and $(\yyy_n)_{n=0}^\infty$ starting from $\vartheta$ such that $(\xxx_n)$ tends to $\xi$, $(\yyy_n)$ tends to $\eta$, and $(\xi|\eta) = \lim_{n\to \infty} (\xxx_n|\yyy_n)$.

\medskip

We first estimate $\Theta (\xi, \eta)$ from below.
Following Lemma \ref{th3.8}(iii), there exists a positive integer $\ell$ such that $(\xi| \eta) = (\xxx_n|\yyy_n)$ for all $n \geq \ell$.
Let $T_\xi=S_{\xxx_\ell}(K) \cup (\bigcup_{\yyy \sim_h \xxx_\ell}S_\yyy(K))$. For $\zzz \in \iota^{-1}(T_\xi)$ with $|\zzz| \geq \ell$, let $\zzz_\ell \in {\mathcal J}_\ell \cap \pi(\vartheta, \zzz)$. Then $d(\zzz_\ell, \xxx_\ell) =0$ or $1$.  By Lemma \ref{th3.8}(ii), (iii) and the triangle inequality (see \eqref{eq3.5'}), we have
\begin{align} \label{eq6.5}
(\zzz|\eta) \geq (\zzz_\ell|\yyy_\ell)
%= (\zzz_{m_0}|\yyy_{m_0})\nonumber\\
%& \geq (\xxx_{m_0}|\yyy_{m_0})-\dfrac{d(\xxx_{m_0}, \zzz_{m_0})}{2} \nonumber\\
\geq (\xxx_\ell|\yyy_\ell) - \dfrac{1}{2} = (\xi|\eta)-\dfrac{1}{2}.
\end{align}
As $K$ is identified with ${\mathcal M}$ (Theorem \ref{th5.1}), Lemma \ref{th6.1} implies that $\iota(Z_n)$ converges to $\xi$ $\mathbb P_\vartheta^\xi$-almost surely.
Let ${t}_\xi = \sup \{n \geq \ell: \iota(Z_n) \notin T_\xi\}+1$, i.e., the first time that $Z_n$ enters $\iota^{-1}(T_\xi)$ and then stays inside. Clearly $\mathbb{P}_\vartheta^\xi({t}_\xi < \infty) = 1$. Therefore we can choose an integer $m_1$ such that $\mathbb{P}_\vartheta^\xi({t}_\xi \leq \tau^*_{m_1})\geq 1/2$. By Lemma \ref{th6.2}, \eqref{eq6.4} and \eqref{eq6.5},
\begin{align*}
\Theta(\xi, \eta)
&\geq {\sum}_{\zzz \in \mathcal{J}_{m_1} \cap\,\iota^{-1}(T_\xi)} \ell_{m_1}^\xi(\zzz)\Theta(\zzz,\eta) \\
&\geq c_1 (\lambda r^\alpha)^{-(\xi|\eta)} ~{\sum}_{\zzz \in \mathcal J_{m_1} \cap\,\iota^{-1}(T_\xi)} \ell_{m_1}^\xi(\zzz) \\
&\geq c_1(\lambda r^\alpha)^{-(\xi|\eta)} \cdot \mathbb{P}_\vartheta^\xi(t_\xi \leq \tau^*_{m_1}) \geq c_2 (\lambda r^\alpha)^{-(\xi|\eta)}.
\end{align*}

\vspace{2mm}

To obtain the upper bound, we first observe that $
\ell_k^\xi(\zzz) \leq G^\xi(\vartheta,\zzz) = G(\vartheta,\zzz)K(\zzz,\xi).
$
This together with the reversibility $G(\vartheta, \zzz) =  (m(\zzz)/m(\vartheta))G(\zzz,\vartheta)$ imply
\begin{align*}
\Theta(\xi, \eta)
&\leq (m(\vartheta))^{-1} \lim\limits_{k \to \infty} \sum\limits_{\zzz \in \mathcal{J}_k} m(\zzz)K(\zzz,\xi)K(\zzz,\eta) \\
&\asymp \lim\limits_{k \to \infty} \sum\limits_{\zzz \in \mathcal{J}_k} (\lambda r^\alpha)^{|\zzz|-(\zzz|\xxx_k)-(\zzz|\yyy_k)}
% &= \lim\limits_{k \to \infty} \sum\limits_{\zzz \in \mathcal{J}_k} ({\lambda r^\alpha})^{(d(\zzz,\xxx_k)+d(\zzz,\yyy_k))/2-k}\\
\end{align*}
(the $\asymp$ is a consequence of  Proposition \ref{th4.5}, Theorem \ref{th5.3} and Lemma \ref{th3.8}(ii)). Let $\Phi(\zzz)$ denote the summand in the above sum. Using Lemma \ref{th3.8}(iii) and the triangle inequality, we have
\begin{equation} \label{eq6.6}
\Phi(\zzz) \leq ({\lambda r^\alpha})^{d(\xxx_k,\yyy_k)/2-k} = (\lambda r^\alpha)^{-(\xxx_k|\yyy_k)} \leq (\lambda r^\alpha)^{-(\xi|\eta)}, \quad \zzz \in \mathcal{J}_k.
\end{equation}
For any ${\www} \in \mathcal{J}_k$, $k \geq 1$, let
$
\mathcal{F}_k(\www) = \{\zzz \in \mathcal{J}_k: d(\zzz,\www) \leq M\},
$
where $M$ is the maximal length of the horizontal geodesics. As $(X,{\mathfrak E})$ is of bounded degree (Lemma \ref{th3.7}), it is direct to check that $\#\mathcal{F}_k({\bf w}) \leq C_1$ for some $C_1>0$ independent of ${\bf w}$ and $k$.
Our aim is to show that for sufficient large $n$, the sum $\sum_{\zzz \in \mathcal{J}_{k+n}} \Phi(\zzz)$ is concentrated on the $n$-th descendants of  $\mathcal{F}_k(\xxx_k) \cup \mathcal{F}_k(\yyy_k)$. This enables us to obtain the desired upper bound.

\vspace{1mm}

To this end, for $\zzz \in \mathcal J_k$, we let ${\mathcal J}_n(\zzz)$ denote the set of $n$-th descendants of $\zzz$ in ${\mathcal J}_{k+n}$. For $\zzz \notin \mathcal F_k(\xxx_k)$ and $\zzz^{(1)} \in \mathcal J_1(\zzz)$, it is clear that $d(\zzz^{(1)},\xxx_{k+1}) \geq d(\zzz,\xxx_k) > M$, which yields $\zzz^{(1)} \notin \mathcal F_{k+1}(\xxx_{k+1})$. Moreover, by the same argument as in Lemma \ref{th3.8}(iii), we have $(\zzz^{(1)}|\xxx_{k+1}) = (\zzz|\xxx_k)$.
Let $\mathcal{F}_k^c = \mathcal{J}_k \setminus (\mathcal{F}_k(\xxx_k) \cup \mathcal{F}_k(\yyy_k))$. It follows that $\Phi(\zzz^{(1)}) = \lambda r^\alpha\Phi(\zzz)$ for any $\zzz \in \mathcal{F}_k^c$, and hence inductively,
 \begin{equation} \label {eq6.7}
 \Phi(\zzz^{(n)}) = (\lambda r^\alpha)^n \Phi(\zzz),\qquad \zzz \in \mathcal{F}_k^c, \ \ \zzz^{(n)} \in {\mathcal J}_n(\zzz).
 \end{equation}
 Also by the same proof as Lemma \ref{th3.0}, we have
\begin{equation} \label {eq6.8}
 r^{-\alpha n} \leq \#{\mathcal J}_n(\zzz) \leq r^{-\alpha (n+1)}.
\end{equation}
Now choose  $n_0$ such that  $\delta = r^{-\alpha} \lambda^{n_0} <1$. It follows from \eqref{eq6.6}-\eqref{eq6.8} that for any $k \geq 1$,
\begin{align*}
\sum\limits_{\www \in \mathcal{J}_{k+n_0}} \Phi(\www)
&= \Big(\sum\limits_{\zzz \in \mathcal{F}_k^c} \sum\limits_{\www\in {\mathcal J}_{n_0}(\zzz)} + \sum\limits_{\zzz \in \mathcal{F}_k(\xxx_k)} \sum\limits_{\www\in {\mathcal J}_{n_0}(\zzz)}  +
 \sum\limits_{\zzz \in \mathcal{F}_k(\yyy_k)} \sum\limits_{\www\in {\mathcal J}_{n_0}(\zzz)} \Big ) \Phi(\www)\\
&\leq r^{-\alpha}\lambda^{n_0}\sum\limits_{\zzz \in \mathcal{F}_k^c} \Phi(\zzz) +  r^{-\alpha (n_0+1)}\Big (\sum\limits_{\zzz \in \mathcal{F}_k(\xxx_k)} +\sum\limits_{\zzz \in \mathcal{F}_k(\yyy_k)}\Big )(\lambda r^\alpha)^{-(\xi|\eta)} \\
&\leq \delta \sum\limits_{\zzz \in \mathcal{J}_k} \Phi(\zzz) + C_2 (\lambda r^\alpha)^{-(\xi|\eta)}.
\end{align*}
By iteration, we have for any integer $q \geq 1$, and any $k \in \{1,2,\cdots,n_0\}$,
\begin{align*}
\sum\limits_{\www \in \mathcal{J}_{k+n_0 q}} \Phi(\www)
&\leq \delta^q \sum\limits_{\zzz \in \mathcal{J}_k}\Phi(\zzz) + \dfrac{C_2}{1-\delta}(\lambda r^\alpha)^{-(\xi|\eta)} \\
&\leq \left(\# \mathcal J_k+\dfrac{C_2}{1-\delta}\right)\cdot (\lambda r^\alpha)^{-(\xi|\eta)} \qquad \hbox{(by \eqref{eq6.6})}\\
&\leq \left(\# \mathcal J_{n_0}+\dfrac{C_2}{1-\delta}\right)\cdot (\lambda r^\alpha)^{-(\xi|\eta)}.
\end{align*}
This completes the upper estimation of $\Theta(\xi,\eta)$.
\end{proof}

\medskip

We remark that for homogeneous IFS, the SRW is a special example of $\lambda$-NRW with $\lambda = 1/N = r^\alpha$. Hence by Theorem \ref{th6.4}, $\Theta(\xi, \eta) \asymp |\xi - \eta|^{-2\alpha}$. Also for the example in Section ~\ref{sec:3}, Remark 2 (see Figure ~\ref{fig:3}), the above theorem (with slight adjustment) shows that for the SRW there,  the Na\"{i}m kernel $\Theta(\xi, \eta) \asymp |\xi - \eta|^{-2}$, which resembles the classical case of the kernel in the Douglas integral \eqref{eq1.2} on the unit circle.

\medskip

We also remark that there is an analogous result for the more general random walks on $N$-ary trees investigated by Kigami \cite{Ki2}.  For the $\lambda$-QNRW on the $N$-homogeneous tree $T$ defined by a set of probability weights $\{p_i\}_{i=1}^N$, the effective resistance between $\xxx \in T$ and $\Sigma_\xxx$ with respect to the sub-tree $T_\xxx$ is $R_\xxx = \lambda^{|\xxx|+1}/((1-\lambda)p_\xxx)$, which can be derived from the combined resistance by iteration. Since the hitting distribution $\nu$ is the self-similar measure of $\{p_i\}_{i=1}^N$ on the Cantor set $\Sigma^\infty$ as Martin boundary, we have $\nu(\Sigma^\infty_\xxx)=p_\xxx$. In the notations of \cite {Ki2},  $D_\xxx:=\nu(\Sigma^\infty_\xxx)R_\xxx=\lambda^{|\xxx|+1}/(1-\lambda)$,  $\lambda_\xxx=1/D_\xxx=(1-\lambda)/\lambda^{|\xxx|+1}$, and $N(\xi, \eta) = (\xi|\eta)$. Following \cite[Theorem 5.6]{Ki2}, we can rewrite his jump kernel $J(\xi,\eta)$ for such $\lambda$-QNRW as
\begin{equation} \label{eq6.9}
J(\xi,\eta) = \dfrac{1-\lambda}{2\lambda}\Big(1+\sum_{m=1}^{(\xi|\eta)}\dfrac{1-\lambda}{\lambda^m p_{[\xi,\eta]_m}}\Big) \asymp \lambda^{-(\xi|\eta)}p_{[\xi,\eta]}^{-1}, \quad \xi,\eta \in \Sigma^\infty,
\end{equation}
where $[\xi,\eta]_m$ is the unique word in $\Sigma^m$ such that both $\xi$ and $\eta$ belong to $\Sigma_{[\xi,\eta]_m}$, and $[\xi,\eta]:=[\xi,\eta]_{(\xi|\eta)}$.
In particular, for the $\lambda$-NRW on the $N$-homogeneous tree, $p_\xxx = N^{-|\xxx|}$, and we have the following estimate:
\begin{equation} \label{eq6.10}
J(\xi,\eta) = \dfrac{(1-\lambda)(N-1)}{2(N-\lambda)}+\dfrac{N(1-\lambda)^2(N/\lambda)^{(\xi|\eta)}}{2\lambda(N-\lambda)} \asymp (N/\lambda)^{(\xi|\eta)}=(\lambda r^\alpha)^{-(\xi|\eta)},
\end{equation}
where $\alpha=|\log N/\log r|$ is the Hausdorff dimension of the Cantor set $\Sigma^\infty$. It coincides with Theorem \ref{th6.4}.

\bigskip

\section{Induced Dirichlet forms}
\label{sec:7}

\noindent In this section we use the $\lambda$-NRW and the Na\"{i}m kernel to induce an energy form on $K$, and make some remarks about the known results for it to be a Dirichlet form. First we summarize some general results that hold for all transient reversible random walks.

\medskip

Let $(X, {\mathfrak E})$ be a countably infinite, connected, locally finite graph,   and let $\{Z_n\}_{n=0}^\infty$ be a transient reversible random walk on $(X, E)$ with conductance $c$ and total conductance $m$. The {\it graph energy} of a (real) function $f$ on $X$ is defined by
\begin{equation*}
{\mathcal{E}}_X[f] = \frac{1}{2} \sum_{x, y \in X, x \sim y} c(x, y) (f(x) - f(y))^2.
\end{equation*}
We let ${\mathcal D}_X $, the domain of ${\mathcal{E}}_X$, be the set of $f: X \to \mathbb R$ with ${\mathcal{E}}_X[f] <\infty$. Since $P(x, y) = c(x, y)/m(x)$, we have
\begin{align} \label{eq7.1}
{\mathcal{E}}_X[f] &= \frac{1}{2} \sum_{x \in X} m(x) \sum_{y \in X, y \sim x} P(x, y)(f(y) - f(x))^2 \nonumber \\
                   &= \frac{1}{2} \sum_{x \in X} m(x) {\mathbb{E}}_x [(f(Z_1) - f(Z_0))^2].
\end{align}
The graph energy defines a  non-negative definite symmetric form ${\mathcal{E}}_X (f, g)$ on ${\mathcal D}_X$ by polarization. Moreover, if we fix any  $x_0 \in X$, then ${\mathcal D}_X$ is a Hilbert space under the inner product
$$
\langle f, g\rangle_0 := f(x_0)g(x_0) + {\mathcal E}_X(f, g), \quad f, g \in {\mathcal D}_X,
$$
and the convergence in $({\mathcal D}_X, {\mathcal E}_X)$ implies pointwise convergence \cite[Lemma 2.4]{Wo1}.

\medskip

\begin{proposition} \label{th7.1} Let ${\mathcal M}$ be the Martin boundary of $\{Z_n\}_{n=1}^\infty$, and let $\nu=\nu_\vartheta$ be the hitting distribution.
Suppose $f \in {\mathcal D}_X$ and is harmonic. Then  $\{f(Z_n)\}_{n = 0}^{\infty}$ converges almost surely and in $L^2$, and there exists $u \in L^2({\mathcal M}, \nu)$ such that  $\lim_{n\to \infty} f(Z_n) =u(Z_\infty)$.  Moreover,  $u$ is uniquely determined $\nu$-a.e., and $f = Hu$ where
$$
Hu (x) = \int_{\mathcal M} u(\xi) K(x, \xi)\,d\nu(\xi), \quad x \in X.
$$
\end{proposition}

\begin{proof}
The first statement was actually proved in \cite[Theorem 1.1]{ALP} without the assumption that $f$ is harmonic. We only need this result when $f$ is harmonic, and the proof is easy for this case. Fix $x \in X$, we claim that  ${\mathbb{E}}_x[f(Z_n)^2]$ is bounded. In fact, Since $f$ is harmonic, $\{f(Z_n)\}_{n=0}^\infty$ is a martingale under ${\mathbb{P}}_x$. It follows that
\begin{align*}
{\mathbb{E}}_x[f(Z_n)^2] &\leq f(x)^2 + \sum_{k = 1}^{\infty} {\mathbb{E}}_x\big [(f(Z_k) - f(Z_{k-1}))^2\big ] \\
&= f(x)^2 + \sum_{k = 1}^{\infty} \sum_{y \in X} P^k(x, y) {\mathbb{E}}_y\big [(f(Z_1) - f(Z_0))^2\big ] \\
&= f(x)^2 + \sum_{y \in X} G(x, y) {\mathbb{E}}_y\big [(f(Z_1) - f(Z_0))^2\big ].
\end{align*}
Using the reversibility, we have $G(x, y) = \frac{m(y)}{m(x)}G(y, x)  \leq \frac{m(y)}{m(x)}G(x, x) $.  Therefore, by (\ref{eq7.1}),
\begin{align*}
{\mathbb{E}}_x[f(Z_n)^2] %\leq f(x)^2 + \frac{G(x, x)}{m(x)} \sum_{y \in X} m(y) {\mathbb{E}}_y[(f(Z_1) - f(Z_0))^2] \\
% &
\leq f(x)^2 + 2 \frac{G(x, x)}{m(x)} \, {\mathcal{E}}_X[f],
\end{align*}
and the claim follows.
Consequently, $\{f(Z_n)\}_{n=0}^\infty$ is an $L^2$-bounded martingale under ${\mathbb{P}}_x$. The martingale convergence theorem implies the almost surely convergence and $L^2$-convergence.

\vspace {0.2cm}

Let $Y = \lim_{n \rightarrow \infty} f(Z_n)$. Note that $Y$ is a {\it final} random variable, i.e. $Y \circ \theta = Y$ a.s. where $\theta$ is the shift operator on the path space. Hence there exists a measurable function $u$ on ${\mathcal{M}}$ such that $Y = u(Z_{\infty})$ a.s. (under any ${\mathbb{P}}_x$) (see \cite{Dy}). It is clear that $u$ is unique $\nu$-a.e. Now since $Y \in L^2({\mathbb{P}}_\vartheta)$, we have
\[
\int_{{\mathcal{M}}} u(\xi)^2 d\nu(\xi) = {\mathbb{E}}_\vartheta(u(Z_{\infty})^2) = {\mathbb{E}}_\vartheta(Y^2) < \infty.
\]
 That $f= Hu$ follows from
\begin{equation*}
f(Z_n)=\mathbb E[u(Z_\infty) | \mathcal F_n]= \mathbb E[u(Z_\infty) | Z_n] = \mathbb E_{Z_n}[u(Z_\infty)] = Hu(Z_n),
\end{equation*}
and the irreducibility of the chain.
\end{proof}

\medskip

 For a reversible random walk on $(X, {\mathfrak E})$, the energy form induces a bilinear form $(\en_{\mathcal M}, \dom_{\mathcal M})$ on $L^2({\mathcal M}, \nu)$ defined by
\begin{equation} \label{eq7.2}
\en_{\mathcal M}(u,v) = \en_X(Hu,Hv),  \quad u, v \in \dom_{\mathcal M} ,
\end{equation}
where  $\dom_{\mathcal M} = \{ u \in L^2({\mathcal M}, \nu):  Hu \in {\mathcal D}_X \}$ is the domain of $\en_{\mathcal M}$. It follows from Proposition \ref{th7.1} that
\begin{equation} \label{eq7.2'}
\dom_{\mathcal M} = \{u \in L^2(\mathcal M, \nu): \ \exists \hbox{ harmonic } f \in {\mathcal D}_X \ \hbox{such that}\  u(Z_\infty)=\lim_{n \to \infty} f(Z_n) \ \hbox{a.s.}\}.
\end{equation}

\medskip

\begin{theorem} \label{th7.2}{\rm \cite[Theorem 3.5]{Si}}
The induced bilinear form $({\mathcal E}_{\mathcal M}, {\mathcal D}_{\mathcal M})$ has the expression
\begin{eqnarray*}
{\mathcal E}_{\mathcal M}[u] = {\mathcal E}_X[Hu] = \frac{1}{2}  m(\vartheta) \int_{\mathcal M} \int_{\mathcal M} (u(\xi) - u(\eta))^2 \Theta(\xi, \eta) d\nu(\xi) d\nu(\eta), \quad u \in {\mathcal D}_{\mathcal M},
\end{eqnarray*}
where $\Theta (\xi,\eta) $ is the Na\"{i}m kernel defined in \eqref{eq6.3}. Moreover,
$ {\mathcal D}_{\mathcal M}
 = \{u \in L^2({\mathcal M}, \nu):\, {\mathcal E}_{\mathcal M} [u] < \infty\} $.
\end{theorem}

\medskip

As a direct consequence of Theorem \ref{th6.4} and the above, we have

\medskip

\begin {theorem} \label {th7.3} Let $\{S_i\}_{i=1}^N$ be an IFS on $\mathbb{R}^d$ that satisfies the OSC, and let $K$ be the self-similar set. Let $\{Z_n\}_{n=0}^{\infty}$ be a $\lambda$-NRW  on the augmented tree $(X, {\mathfrak E})$. Then
\begin{eqnarray*}
{\mathcal E}_K[u] \asymp  \int_K \int_K (u(\xi) - u(\eta))^2 \ |\xi - \eta|^{-(\alpha + \beta)} d\nu(\xi) d\nu(\eta), \quad u \in {\mathcal D}_K,
\end{eqnarray*}
where $\alpha=\dim_H K$, and $\beta = \frac{\log \lambda}{\log r}$.
\end{theorem}

\medskip

Next we recall the definition of Dirichlet form (see \cite{CF,FOT}). Let $\mathfrak{X}$ be a locally compact separable metric space together with a positive Radon measure  $\mu$ such that $\supp (\mu) = \mathfrak{X}$; also let $C_0(\mathfrak{X})$ denote the space of continuous function on $\mathfrak{X}$ with compact support.

\medskip

\begin{definition} \label{th7.4}
A Dirichlet form $(\en, \dom)$ on $L^2(\mathfrak{X}, \mu)$ is a bilinear form which is symmetric, non-negative definite, closed, Markovian and densely defined on $L^2(\mathfrak{X}, \mu)$. It is called {\it regular} if the subspace $\dom \cap C_0(\mathfrak{X})$ is dense in $\dom$ with the $\en_1$-norm, and is dense in $C_0(\mathfrak{X})$ with the supremum norm. It is called {\it local} if for two functions $u,v \in \dom$ having disjoint compact supports, $\en(u,v)=0$.

\end{definition}

\medskip

It is easy to see that the bilinear form ${\mathcal E}_K(\cdot, \cdot)$ in Theorem \ref {th7.3} is symmetric, non-negative definite, closed and Markovian. For it to be a Dirichlet form, we need to show that its domain ${\mathcal D}_K$ is dense in $L^2(K, \nu)$. In view of the fact that the Na\"{i}m kernel in Theorem \ref{th7.3} satisfies the estimate $\Theta (\xi, \eta) \asymp |\xi - \eta|^{-(\alpha + \beta)}$, we introduce the following Besov space on $L^2(\mathfrak{X},\mu)$ (see \cite{J,St,GHL1,GHL2,HK}).

\medskip

For convenience, we assume that ${\mathfrak X}\subset {\mathbb R}^d$ and is equipped with the Euclidean distance; we also assume that ${\mathfrak X}$ has Hausdorff dimension $\alpha$, and $\mu (B(x;r)) \asymp r^\alpha$ for all $x \in {\mathfrak X}$ and $0<r<1$. We call such ${\mathfrak X}$ an {\it $\alpha$-set} \cite{JW}. For $\sigma >0$, and for $u \in L^2({\mathfrak X},\mu)$, we define
\begin{equation} \label{eq7.4}
\norm_{2,2}^{\alpha, \sigma}(u) = \ds\int_0^\infty \dfrac{dr}{r} \dfrac{1}{r^{\alpha+2\sigma} }\ds\iint\limits_{\{\xi,\eta \in {\mathfrak X}: |\xi-\eta|<r\}} (u(\xi)-u(\eta))^2d\mu(\xi)d\mu(\eta),
\end{equation}
and the {\it Besov spaces}
$
\Lambda_{2,2}^{\alpha, \sigma} = \{u \in L^2({\mathfrak X},\mu): \norm_{2,2}^{\alpha, \sigma}(u) < \infty\}
$
with the associated norms
$
\Vert u \Vert_{\Lambda_{2,2}^{\alpha, \sigma}}^2 = \Vert u \Vert_2^2 + \norm_{2,2}^{\alpha, \sigma}(u) $.
 The space can be trivial when $\sigma$ is a large value.  (For example, in Euclidean space $\mathbb{R}^d$, $\Lambda_{2,2}^{d,1} = \{0\}$.)
We introduce an important quantity $\beta^*\in [0, +\infty]$  which is intrinsic to the underlying space $\mathfrak X$ \cite{St,GHL1}:
\begin{align} \label{eq7.7}
\beta^* := \sup \{\beta>0: \Lambda_{2,2}^{\alpha, \beta/2} \cap C_0(\mathfrak X) \hbox{ is dense in } C_0(\mathfrak X) \}.
\end{align}
It is called the {\it critical exponent} of the family $\left\{\Lambda_{2,2}^{\alpha, \beta/2}\right\}_{\beta>0}$. The value of $\beta^*$ is already known for some standard cases:  for the Euclidean space $\mathbb{R}^d$, we have $\alpha =d, \beta^*=2$. For $d$-dimensional Sierpi\'{n}ski gasket, then $\alpha = \log (d+1)/ \log 2$, and $\beta^*=\log (d+3)/ \log 2$ \cite{J}. There are also extensions to nested fractals and related Besov spaces \cite{P1,P2}, and  evaluation of some other specific cases \cite{Ku}. For Cantor-type set as the boundary of an infinite binary tree, it follows from \cite[Theorem 5.6]{Ki2} that $\beta^*=\infty$.  In general, we know that if a metric measure space $({\mathfrak X}, \mu)$ satisfies $\mu (B(x; r)) \asymp r^{\alpha}$ for all $x \in {\mathfrak X}$ and $0<r<1$, then $\beta^* \geq 2$. If in addition  $\mathfrak{X}$ satisfies {\it the chain condition} \cite{GHL1}, then $\beta^* \leq \alpha+1$.

\medskip

Continuing the statement in Theorem \ref{th7.3}, we have the following conclusion.

\medskip

\begin{theorem} \label{th7.8}
With the same assumption and notations as in Theorem \ref{th7.3}, we have
\begin{equation} \label {eq7.8}
\en_K[u] \asymp \norm_{2,2}^{\alpha, \beta/2}(u), \quad u \in \dom_K ,
\end{equation}
and $\dom_K$ is the Besov space $\Lambda_{2,2}^{\alpha, \beta/2}$. Therefore, if $\beta < \beta^*$, then $(\en_K, \dom_K)$ is a non-local Dirichlet form on $L^2(K,\nu)$.
\end{theorem}

\medskip
The proof of \eqref{eq7.8} and that $\dom_K$ is the Besov space $\Lambda_{2,2}^{\alpha, \beta/2}$ are in \cite{St}. The following proposition deals with the regularity of the induced Dirichlet form.

\medskip

\begin{proposition} \label{th7.9}
Let $\{S_i\}_{i=1}^N$ be an IFS of contractive similitudes with the OSC, \  and let $K$ be the self-similar set. For a $\lambda$-NRW on the augmented tree $(X, {\mathfrak E})$, if either (i) $\lambda \in (r^2,1)$, or (ii) $\alpha < \beta^*$ and $\lambda \in (r^{\beta^*},r^\alpha)$, then the induced form $({\mathcal E}_K, {\mathcal D}_K)$ is a regular non-local Dirichlet form on $L^2(K,\nu)$.

\vspace{0.1cm}

 Moreover, if $2 \leq \alpha<\beta^*$ and $\lambda \in [r^\alpha,r^2]$, let ${\mathcal D}_K^* = \overline {C(K) \cap {\mathcal D}_K}$, where the closure is taken under the norm $||\cdot||_{\Lambda_{2,2}^{\alpha, \beta/2}}$. Then $({\mathcal E}_K, {\mathcal D}_K^*)$ is a regular non-local Dirichlet form on $L^2(K,\nu)$.
\end{proposition}

\begin{proof} It is well-known that $K$ equipped with Euclidean metric and $\alpha$-Hausdorff measure is an $\alpha$-set. As $\beta^* \geq 2$, the assumption $\lambda \in (r^2, 1)$ in (i) implies $0< \beta< 2 (\leq \beta^*)$. It follows from \cite[Theorem 3]{St} that the Dirichlet form $(\en_K, \dom_K)$ is regular.

\vspace {0.1cm}

For (ii), we have $\alpha <\beta< \beta^*$, and the domain ${\mathcal D}_K$ is embedded into the Lipschitz space of H\"older exponent $(\beta -\alpha)/2$ (see \cite[Theorem 4.11]{GHL1}). Hence $({\mathcal E}_K, {\mathcal D}_K)$ is regular.

For the last part, we have  $2\leq \beta \leq \alpha < \beta^*$.  Let $\beta_0 = (\alpha + \beta^*)/2$. Then $\alpha < \beta_0 < \beta^*$. Hence by the above paragraph, ${\mathcal D}_K^{(\beta_0)}$ consists of certain Lipschitz functions of order $(\beta_0 -\alpha)/2$, and
$
 {\mathcal D}_K^{(\beta_0)} = C(K) \cap {\mathcal D}_K^{(\beta_0)} \subset C(K) \cap {\mathcal D}_K \subset {\mathcal D}_K^*.
$
This implies that $C(K) \cap {\mathcal D}_K$ is dense in $L^2(K, \nu)$ and in $C(K)$, and $({\mathcal E}_K, {\mathcal D}_K^*)$ is regular.
\end{proof}

\medskip

\noindent {\it Remark.} In the proof of the last part, we cannot prove $ C(K)\cap {\mathcal D}_K$ is dense in $\dom_K$ with the $\en_1$-norm. Hence we use $ {\mathcal D}_K^\ast = \overline {C(K)\cap {\mathcal D}_K}$ instead of the original $\dom_K$. Also we do not know the regularity of the Dirichlet form for $2 \leq \beta< \beta^* \leq \alpha$.

\medskip

By \cite[Theorems 7.2.1, 7.2.2]{FOT}, we know that a regular Dirichlet form $(\mathcal E_K, \mathcal D_K)$ on $L^2(K,\nu)$ generates an associated Hunt jump process with transition density function (heat kernel) $p(t,\xi,\eta)$. For $0<\beta<2$ , i.e., $\lambda \in (r^2,1)$, it has been shown in \cite{CK} that the heat kernel satisfies the following estimate:
$$
p(t,\xi,\eta) \asymp \min \Big \{ t^{-\alpha/\beta},\, \frac{t}{|\xi-\eta|^{\alpha+\beta}} \Big \}, \quad \xi,\eta \in K, \ 0< t \leq 1.
$$
We do not have estimates of $p(t,\xi,\eta)$ for $\alpha<\beta<\beta^*$ or for other cases in general.

\medskip

The Besov spaces at the critical exponent $\beta^*$ are particularly important, and are not completely understood. In fact, there is another class of Besov spaces $\Lambda_{2, \infty}^{\alpha, \sigma}$ involved (see \cite{GHL1,GHL2}). For example on ${\mathbb R}^d$, $\beta^* =2$, and we have $\Lambda_{2,2}^{d, 1}({\mathbb R}^d) = \{0\}$, but $\Lambda^{d, 1}_{2,\infty}$ equals the Sobolev space $W^1_2 ({\mathbb R}^d)$, the domain of classical Dirichlet form that generates the Gaussian heat kernel. Similar situations hold for self-similar sets (and more general metric measure spaces) that admit local Dirichlet forms and subgaussian kernels \cite{GHL1,GHL2}. In forthcoming paper \cite{KL}, we will give a more detailed discussion, and provide a criterion to determine the exponent $\beta^*$.

\bigskip

\noindent {\bf Acknowledgements}: The authors would like to thank Professors A.~Grigor'yan, J.X.~Hu, J.~Kigami and T.~Kumagai for many valuable discussions and suggesting some references, and to  Professor S.M.~Ngai for going through the manuscript carefully. They are also indebted to the referee for  many constructive comments which helped improve the presentation of the paper. Part of the work was carried out while the second author was visiting the University of Pittsburgh, he is grateful to Professors C.~Lennard and J.~Manfredi for the arrangement of the visit.

\bigskip
\bigskip

\noindent SHI-LEI KONG, Department of Mathematics, The Chinese University of Hong Kong, Hong Kong\\
slkong@math.cuhk.edu.hk

\bigskip
\noindent KA-SING LAU, Department of Mathematics, The Chinese University of Hong Kong, Hong Kong\\
kslau@math.cuhk.edu.hk

\bigskip
\noindent TING-KAM LEONARD WONG, Department of Mathematics, University of Southern California, Los Angeles, CA 90089 USA\\
tkleonardwong@gmail.com

\end{document}